\documentclass[11pt]{amsart}

\usepackage{amsfonts}
\usepackage{amssymb}
\usepackage{graphicx}
\usepackage{amsmath}
\usepackage{latexsym}
\usepackage{amscd}
\usepackage{xypic}
\usepackage[all,cmtip]{xy}
\usepackage{mathrsfs}
\usepackage{enumitem} 
\usepackage{braket}
\usepackage[margin=1.2in]{geometry}

\usepackage{tikz} \usetikzlibrary{arrows,snakes,shapes,decorations.markings,patterns}
\usepackage{imakeidx}
\makeindex[title=Index of Notation]

\usepackage{hyperref}

\allowdisplaybreaks 

\numberwithin{equation}{section}

\newtheorem{prop}{Proposition}[section]
\newtheorem{theo}[prop]{Theorem}
\newtheorem{lemm}[prop]{Lemma}
\newtheorem{coro}[prop]{Corollary}

\newtheorem*{claim*}{Claim}

\theoremstyle{definition}

\newtheorem{defi}[prop]{Definition}

\newcommand{\CC}{\mathbb{C}}

\newcommand{\NN}{\mathbb{N}}

\newcommand{\RR}{\mathbb{R}}
\renewcommand{\SS}{\mathbb{S}}

\newcommand{\cC}{\mathcal C}

\renewcommand{\cH}{\mathcal H}

\renewcommand{\cL}{\mathcal L}
\newcommand{\cM}{\mathcal M}
\newcommand{\cN}{\mathcal N}

\newcommand{\cS}{\mathcal S}

\newcommand{\br}{\mathbf{r}}
\newcommand{\bR}{\mathbf{R}}

\DeclareMathOperator{\supp}{supp}

\newcommand{\bangle}[1]{\left\langle #1 \right\rangle}

\DeclareMathOperator{\Ric}{Ric}

\define{\sff}{{\rm II}}
\define{\tfsff}{\accentset{\circ}{\sff}}

\let\oldmarginpar\marginpar
\renewcommand\marginpar[1]{\-\oldmarginpar[\raggedleft\footnotesize #1]%
{\raggedright\footnotesize #1}}

\DeclareMathOperator{\Jac}{Jac}

\DeclareMathOperator{\Id}{Id}

\DeclareMathOperator{\proj}{proj}

\DeclareMathOperator{\Graph}{graph}

\newcommand{\eps}{\varepsilon}

\usepackage{graphicx}

\allowdisplaybreaks

\title{Uniqueness of asymptotically conical tangent flows}
\author{Otis Chodosh}
\address{OC: Department of Mathematics, Bldg.\ 380, Stanford University, Stanford, CA 94305, USA}
\email{ochodosh@stanford.edu}

\author{Felix Schulze}
\address{FS: Department of Mathematics, Zeeman Building, University of Warwick, Gibbet Hill Road, Coventry CV4 7AL,
UK}
\email{felix.schulze@warwick.ac.uk} 

\date{\today}

\begin{document}

\maketitle

\begin{abstract} 
Singularities of the mean curvature flow of an embedded surface in $\RR^{3}$ are expected to be modeled on self-shrinkers that are compact, cylindrical, or asymptotically conical. In order to understand the flow before and after the singular time, it is crucial to know the uniqueness of tangent flows at the singularity. 

In all dimensions, assuming the singularity is multiplicity one, uniqueness in the compact case has been established by the second-named author \cite{Schulze:compact}, and in the cylindrical case by Colding--Minicozzi \cite{CM:uniqueness}. We show here the uniqueness of multiplicity-one asymptotically conical tangent flows for mean curvature flow of hypersurfaces.

In particular, this implies that when a mean curvature flow has a multiplicity-one conical singularity model, the evolving surface at the singular time has an (isolated) regular conical singularity at the singular point. This should lead to a complete understanding of how to ``flow through'' such a singularity. 
\end{abstract}

\section{Introduction}

\subsection{Uniqueness of tangent flows} By work of Huisken \cite{Huisken:singularities}, White \cite{White:announcement}, and Ilmanen \cite{ilmanen:surface-sing}, singularities of mean curvature flow can be modeled by self-similar shrinking solutions to the flow. For flows of embedded surfaces in $\RR^{3}$, Ilmanen proves \cite{ilmanen:surface-sing} that self-shrinkers arising as tangent flows at the first singular time are smooth and embedded (possibly with higher multiplicity). Wang \cite{Wang:ends-conical} has proven that such shrinkers, if non-compact, have ends that are asymptotic to a cylinder or smooth cone (cf.\ Definition \ref{defi:as-con}); see also \cite{SunWang:compactness}. Moreover, Kapouleas--Kleene--M\o ller \cite{KKM:shrinkers} and Nguyen \cite{Nguyen1,Nguyen2,Nguyen3} have constructed embedded, smooth, self shrinkers in $\RR^{3}$ with (smoothly) conical ends.

An important question is to determine whether or not these tangent flows are unique. The second-named author has proved \cite{Schulze:compact} that this holds (in all dimensions and co-dimension) when there is a compact, multiplicity one, (smooth) tangent flow. Colding--Minicozzi \cite{CM:uniqueness} (cf.\ \cite{ColdingIlmanenMinicozzi}) have proven that uniqueness holds (for hypersurfaces, in all dimensions) for multiplicity one cylindrical tangent flows; see also   \cite{BernsteinWang:uniqueness-mult}.

In this work, we show that uniqueness also holds in the case of multiplicity one tangent flows whose self shrinker is smoothly conical.

\begin{theo}[Uniqueness of conical tangent flows]\label{thm:unique}
Fix $\Sigma^{n}\subset \RR^{n+1}$\index{$\Sigma$} an asymptotically conical self-shrinker. Let $\cM = (\mu_{t})_{t\in (-t_{1},0)}$ be an integral $n$-Brakke flow so that the self-similar shrinking multiplicity one Brakke flow associated to $\Sigma$, $\cM_{\Sigma}$, arises as a tangent flow to $\cM$ at $(0,0)$. Then $\cM_{\Sigma}$ is the unique tangent flow to $\cM$ at $(0,0)$. 
\end{theo}

See Section \ref{subsec:rate-convergence} for estimates concerning the rate of convergence. We expect that the argument will extend to higher codimension with little change. 

 An interesting feature of our proof of Theorem \ref{thm:unique} is that it shows that the \L ojasiewicz--Simon approach to uniqueness of blow-ups can be applied in the case of a non-compact singularity model. Colding--Minicozzi's work on the uniqueness of cylindrical tangent flows \cite{CM:uniqueness} does not proceed via a reduction to the finite dimensional \L ojasiewicz inequality \`a la Simon, but rather proves a \L ojasiewicz-type inequality by hand, using the explicit structure of the cylinder in a fundamental way. Here the situation is different: we do not use any explicit structure of the conical shrinkers, so instead must rely on a \L ojasiewicz--Simon inequality proven by ``abstract'' methods, after introducing relevant weighted function spaces. 
 
 This approach has the drawback that it requires much stronger ``closeness'' of the flow relative to the shrinker. Thus, we must develop a new ``extension of closeness'' mechanism that is not present in the cylindrical case (cf.\ Lemma \ref{lemm:model-problem} and Proposition \ref{prop:approx-up-to-rough}). We then must combine this mechanism with several crucial ideas of Colding--Minicozzi concerning improvement and extension of curvature estimates to overcome the non-compactness of the problem. 
 
Our approach seems to be quite general and flexible; we expect that it will apply to the uniqueness of non-compact singularities in other geometric problems, when the singularity is ``well behaved'' at infinity. 

\subsection{The structure of the singular set around an asymptotically conical shrinker} We note that conjecturally (cf. Ilmanen's no cylinder conjecture \cite[\#12]{Ilmanen:problems}), the cylinder is the only shrinker in $\RR^{3}$ with a cylindrical end. Combing Theorem \ref{thm:unique} with \cite{Schulze:compact}, \cite{CM:uniqueness}, and \cite{Wang:ends-conical}, it would follow that for the mean curvature flow of a smooth embedded surface in $\RR^{3}$, \emph{all} multiplicity one tangent flows at the first singular time are unique. 

Uniqueness of tangent flows gives important information about the singular behavior of the flow. Using their result on the uniqueness of cylindrical tangent flows, Colding--Minicozzi have proven \cite{CM:sing-set} (among other things) that a mean curvature flow of hypersurfaces in $\RR^{n+1}$ with only multiplicity one cylindrical tangent flows has space-time singular set contained in finitely many compact embedded $(n-1)$-dimensional Lipschitz submanifolds and a $(n-2)$-dimensional set. Moreover, in $\RR^{3}$ they have shown that such flows are smooth for almost all times, and any connected component of the singular set is completely contained in a time-slice (see also \cite{ColdingMinicozzi:level-set-flow}). 

Similarly, Theorem \ref{thm:unique} (and the pseudolocality arguments used in Lemma \ref{lemm:rough-improves} below) implies the following 

\begin{coro}\label{coro:conical-struct}
For $\cM$ and $\Sigma$ as in Theorem \ref{thm:unique}, there is $\eps>0$ so that for all $t\in (-\eps^{2},0)$, we have $\mu_{t}\lfloor B_{\eps}(0) = \cH^{n}\lfloor M_{t}$ for a smooth family $M_{t}$ of embedded surfaces flowing by mean curvature in $B_{\eps}(0)$. The surfaces $M_{t}$ are diffeomorphic to $\Sigma$. Moreover,  as $t\nearrow 0$, the flow $M_{t}\cap (B_{\eps}(0)\setminus\{0\})$ converges in $C^{\infty}_{\textnormal{loc}}$ to a smooth surface $M_{0}\subset B_{\eps}(0)\setminus \{0\}$ with a conical singularity at $0$ smoothly modeled\footnote{In other words, rescaling $M_{0}$ around $0$ converges in $C^{\infty}_{\textrm{loc}}(\RR^{n+1}\setminus\{0\})$ to the asymptotic cone of $\Sigma$.} on the asymptotic cone of $\Sigma$. 
\end{coro}

We note that Colding--Minicozzi have proven \cite{CM:generic} that the plane, sphere, and cylinders are the unique entropy stable shrinkers. They have proposed this as a mechanism for a possible way to construct a generic mean curvature flow. Corollary \ref{coro:conical-struct} suggests that one can flow \emph{through} points with conical tangent flows, instead of trying to perturb them away. Understanding the flow through these ``non-generic'' situations will be particularly important towards understanding families of mean curvature flows. We will investigate this elsewhere. 

\subsection{Some recent results in singularity analysis of mean curvature flow} We remark that Brendle has recently proven \cite{Brendle:genus0} that the only smooth properly embedded self shrinkers in $\RR^{3}$ with genus zero are the plane, sphere, and cylinder; hence, a conical shrinker must have non-zero genus. Moreover Bernstein--Wang have shown \cite{BernsteinWang:sharp-lower-bd} that the round sphere has the least entropy among any closed hypersurface (up to the singular dimension, cf.\ \cite{Zhu:entropy} and see also \cite{ColdingIlmanenMinicozziWhite,KetoverZhou}); the same authors have extended \cite{BernsteinWang:topological-AC-shrinkers} this to non-compact surfaces in $\RR^{3}$ (see also \cite{BernsteinWang:top}). Wang has proven \cite{Wang:uniqueness-conical} that two shrinkers asymptotic to the same smooth cone must be identical. Ketover has recently constructed \cite{Ketover:self-shrinkers} self-shrinking Platonic solids.

Brendle--Choi have classified \cite{BrendleChoi} the bowl solition as the unique strictly convex ancient solution in $\RR^{3}$ (cf.\ \cite{Wang:convex,Haslhofer:bowl,Hershkovits:bowl,BrendleChoi:higher-dim}). Moreover, Angenent--Daskalopoulos--Sesum have classified closed non-collapsed ancient solutions that are uniformly two-convex \cite{ADS:ancientOvals}. Finally, Choi--Haslhofer--Hershkovits \cite{ChoiHaslhoferHerskovits} have proven the mean convex neighborhood conjecture in $\RR^{3}$, by classifying low entropy ancient solutions (see also \cite{HershkovitsWhite}). 

\subsection{Idea of the proof of Theorem \ref{thm:unique}} 

The basic idea to prove Theorem \ref{thm:unique} is to rely on a \L ojasiewicz-type inequality (see \cite{Loj,Simon:Loj,Simon:green-book}) to show uniqueness of the tangent flow. Indeed, this strategy was already successful in the compact \cite{Schulze:compact} and cylindrical \cite{CM:uniqueness} cases. In the cylindrical and conical cases, the non-compactness of the shrinker causes serious issues (beyond simply those of a technical nature), due to the fact that one cannot write the entire flow as a graph over the shrinker. 

Unlike the cylindrical case \cite{CM:uniqueness}, we do not exploit any specific structure of the shrinker (beyond the fact that it has conical ends). Conical ends seem to be less degenerate with regards to the uniqueness problem, allowing us to obtain very strong estimates in annular regions around the point where the singularity is forming.
Because we do not assume any specific structure of the shrinker, we must prove the \L ojasiewicz--Simon inequality by ``abstract'' methods (i.e., by a finite dimensional reduction to \L ojasiewicz's original inequality \cite{Loj}). In Section \ref{sec:lin-est-spaces}, we construct weighted H\"older and Sobolev spaces in which Simon's argument \cite{Simon:Loj} can be used  to prove a \L ojasiewicz--Simon inequality for entire graphs over the shrinker (see Theorem \ref{theo:loj-entire}). Roughly speaking, we consider H\"older spaces (inspired by \cite{KKM:shrinkers}) $\cC\cS^{2,\alpha}_{-1}(\Sigma)$ of functions $u:\Sigma\to\RR$ so that in coordinates $(r,\omega) \in  (1,\infty)\times \Gamma $ on the end of $\Sigma$ (where $\Gamma$ is the link of the asymptotic cone of $\Sigma$),
\[
f(r,\omega) = c(\omega) r + O(r^{-1})
\]
where the error term is taken in $C^{2,\alpha}$ on balls of unit size. We also require the improved radial derivative estimate
\[
\partial_{r}f(r,\omega) = c(\omega) +  O(r^{-2})
\]
in $C^{0,\alpha}$. Geometrically, we can think of $\cC\cS^{2,\alpha}_{-1}(\Sigma)$ as functions whose graphs are asymptotically conical (for a different cone) and decay to their asymptotic cone at a rate $O(r^{-1})$ in $C^{2,\alpha}$.

The linearized shrinker operator maps the space $\cC\cS^{2,\alpha}_{-1}(\Sigma)$ to $\cC\cS^{0,\alpha}_{-1}(\Sigma)$, i.e. $Lu = O(r^{-1})$ in $C^{0,\alpha}$ (this is where the improved radial derivative estimate is needed). We can prove Schauder estimates for the $L$ operator between these spaces (see Proposition \ref{prop:CS-schauder}). Moreover (based on ideas communicated to us by J.\ Bernstein \cite{bernstein:private-Schauder}) one can also establish (see Section \ref{subsec:weighted:sob}) regularity and existence for the $L$ operator (the linearized shrinker operator) between $L^{2}$-based Sobolev spaces $L_{W}^{2}(\Sigma)$ and $H^{2}_{W}(\Sigma)$, when weighted by the Gaussian density $\rho = (4\pi)^{-\frac n 2} e^{-|x|^{2}/4}$. Combining these facts, we find that the $L$ operator behaves between these spaces in essentially the same way as in the compact cases considered by Simon \cite{Simon:Loj}. This yields a \L ojasiewicz--Simon inequality for \emph{entire} graphs over $\Sigma$ (Theorem \ref{theo:loj-entire}), i.e., if $\Vert u\Vert_{\cC\cS^{2,\alpha}_{-1}(\Sigma)}$ is sufficiently small, then for $M = \Graph u$,
\begin{equation}\label{eq:intro-Loj-entire}
|F(M) - F(\Sigma)|^{1-\theta} \leq C \left( \int_{M} |\phi|^{2} \rho \, d\cH^{n}\right)^{\frac 12}. 
\end{equation}
Here $F(M)$ is the Gaussian area (see Definition \ref{defi:gaussian-area}) and $\phi$ is the deviation from $M$ being a shrinker (see Definition \ref{defi:phi}). 

To apply \eqref{eq:intro-Loj-entire} to prove uniqueness of conical tangent flows, the basic strategy is to show that if a Brakke flow $\cM$ has a multiplicity one conical tangent flow (modeled by $\Sigma$) at $(0,0)$, then it is possible to write part of $\cM$ as a graph over part of $\Sigma$, and that this graphical function extends to a function that is small in $\cC\cS^{2,\alpha}_{-1}(\Sigma)$. At this point \eqref{eq:intro-Loj-entire} can be applied to this extended function. Applying the resulting inequality to $\cM$ introduces errors based on the fact that $\cM$ is not an entire graph over $\Sigma$. Controlling the size of these errors relative to the terms in \eqref{eq:intro-Loj-entire} is a serious issue, which we now describe in some detail. 

We consider the rescaled mean curvature flow around $(0,0)$; assume the rescaled flow consists of surfaces $M_{\tau}$ for $\tau \in [-1,\infty)$ and $M_{\tau_{i}}\to\Sigma$ in $C^{\infty}_{\textrm{loc}}$ along some sequence $\tau_{i}\to\infty$. We seek to prove by a continuity argument that for $\underline r$ fixed and $\tau$ sufficiently large, $M_{\tau} \cap B_{\underline r}$ is a $C^{\ell+1}$ graph of a function with $C^{\ell+1}$-norm bounded by $b$. This is (roughly) the \emph{core graphical hypothesis} $(*_{b,\underline r})$ (see Definition \ref{def:core-graph-hypoth}). Notice that the core graphical hypothesis will not suffice to control the errors when applying the \L ojasiewicz inequality. The reason for this is that we must not destroy the term
\[
\int_{M_{\tau}} |\phi|^{2} \rho \, d\cH^{n} := e^{-\frac{\bR(M_{\tau})^{2}}{4}}
\]
on the right hand side of \eqref{eq:intro-Loj-entire}. We call $\bR(M_{\tau})$ the shrinker scale (Definition \ref{defi:shrinker-scale}).\footnote{Note that our shrinker scale differs from the definition used in \cite{CM:uniqueness} slightly, due to the nature of our \L ojasiewicz--Simon inequality.} On the other hand, cutting off the \L ojasiewicz--Simon inequality outside of a ball of radius $R$ will introduce terms on the order of $o(1)e^{-\frac{R^{2}}{4}}$ (see Theorem \ref{theo:cutoff-loj}). Thus, we must show that $M_{\tau}$ is graphical over $\Sigma\cap B_{R}$ for $R\sim \bR(M_{\tau})$. More precisely, we must show that there is $u:\Sigma\to\RR$ with $\Vert u \Vert_{\cC\cS^{2,\alpha}_{-1}(\Sigma)}$ sufficiently small so that $M_\tau\cap B_{R}$ is contained in the graph of $u$. We call the largest $R$ satisfying this property the \emph{conical scale} (Definition \ref{defi:conical-scale}), denoted by $\br_{\ell}(M_{\tau})$. We would thus like to show that the  the conical scale $\br_{\ell}(M_{\tau})$ is comparable to the shrinker scale $\bR(M_{\tau})$.

Observe that this is far from clear: we must show that $M_{\tau}$ decays like $O(r^{-1})$ towards a cone (which is close to the asymptotic cone of $\Sigma$) nearly all the way to $\bR(M_{\tau})$. However, if $\bR(M_{\tau})$ is very large, we have to transmit the graphical information contained in the core graphical hypothesis (only on a fixed compact set) essentially all the way to $\bR(M_{\tau})$, while even obtaining decay!

The way we do this has some features in common with the methods used in \cite{CM:uniqueness}, but the argument on the whole is rather different. To obtain control on the conical scale $\br_{\ell}(M_{\tau})$ we first introduce a weaker notion, the \emph{rough conical scale} $\tilde\br_{\ell}(M_{\tau})$ (Definition \ref{defi:rough-conical-scale}), which is the largest radius where the curvature of $M_{\tau}$ behaves like the curvature along a cone. As a preliminary step, we prove that the rough conical scale improves very rapidly, as long as the core graphical hypothesis $(*_{b,\underline r})$ is satisfied.

\begin{figure}[h]
\begin{tikzpicture}
\filldraw [blue!50] (-2.9,-2) rectangle (-1.1,0);
\filldraw [blue!50] (2.9,-2) rectangle (1.1,0);

\draw (-3,0) -- (3,0); 
\filldraw (0,0) circle (1pt) node [above] {$(0,0)$};
\draw plot [smooth, domain = -2:2] (\x,{-(.5)*(\x)^2});
\draw [thick] (-3,-2) -- (-1,-2);
\draw [thick] (1,-2) -- (3,-2);
\draw [dashed] (-1,-2) -- (1,-2);

\draw [->] (-3.5,-2) -- (-3.5,.5) node [above] {$t$};
\filldraw (-3.5,0) circle (1pt) node [left] {$t=0$};
\filldraw (-3.5,-2) circle (1pt) node [left] {$t=-1$};
\draw [->] (-.1,-3) -- (-1.4,-2.1);
\draw [->] (.1,-3) -- (1.4,-2.1);
\node at (0,-3.2) {conical part of $M_{-1}$};
\end{tikzpicture}

\caption{The conical nature of the shrinker $\Sigma$ (and thus the unrescaled flow at time $t=-1$) yields---via pseudolocality---curvature estimates in the region that is shaded blue. Note that we can only expect \eqref{eq:intro-Loj-entire} to give useful bounds \emph{below} the parabola, since this is the set where the backwards heat kernel $\rho$ is uniformly bounded away from zero. }

\label{fig:pseudolocality-1}
\end{figure}
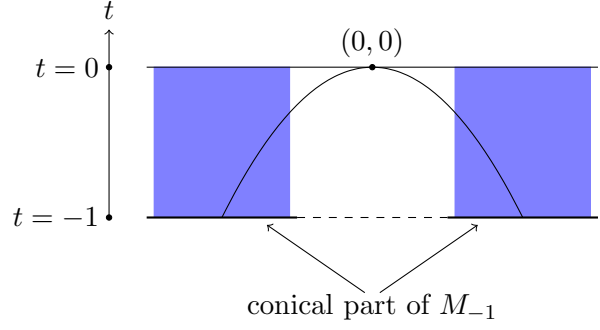

\begin{figure}[h]
\begin{tikzpicture}
\filldraw [blue!50] (-2.9,-2) rectangle (-1.1,0);
\filldraw [blue!50] (2.9,-2) rectangle (1.1,0);

\filldraw [blue!50] (-2.9,-1) rectangle (-.8,0);
\filldraw [blue!50] (2.9,-1) rectangle (.8,0);

\draw (-3,0) -- (3,0); 
\filldraw (0,0) circle (1pt) node [above] {$(0,0)$};
\draw plot [smooth, domain = -2:2] (\x,{-(.5)*(\x)^2});
\draw [thick] (-3,-2) -- (-1,-2);
\draw [thick] (1,-2) -- (3,-2);
\draw [dashed] (-1,-2) -- (1,-2);

\draw plot [smooth, domain = -1:1] (\x,{-2*(\x)^2});
\draw [thick] (-3,-1) -- (-.7,-1);
\draw [thick] (.7,-1) -- (3,-1);
\draw [dashed] (-1,-1) -- (1,-1);

\draw [->] (-3.5,-2) -- (-3.5,.5) node [above] {$t$};
\filldraw (-3.5,0) circle (1pt) node [left] {$t=0$};
\filldraw (-3.5,-2) circle (1pt) node [left] {$t=-1$};
\filldraw (-3.5,-1) circle (1pt) node [left] {$t=-\frac 12$};
\draw [->] (-.1,-3) -- (-1,-1.1);
\draw [->] (.1,-3) -- (1,-1.1);
\node at (0,-3.2) {conical part of $M_{-\frac12}$};
\end{tikzpicture}

\caption{Assuming that we have control over $M_{t}$ via \eqref{eq:intro-Loj-entire} inside of the wide parabola (for $t \in [-1,-\frac 12 )$), we can then use pseudolocality out of the conical region in $M_{-\frac 12}$ to gain curvature estimates on a \emph{larger} region (still shaded blue). This is our first improvement/iteration mechanism.}

\label{fig:pseudolocality-2}
\end{figure}
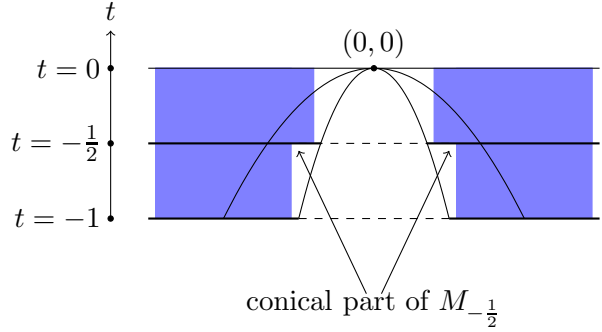

Indeed, to control the rough conical scale $\tilde\br_{\ell}(M_{\tau})$ we first observe that pseudolocality applied to the unrescaled flow gives curvature estimates on an annular region that persist all the way up to the singular time (using the fact that the flow is close on a large compact set to the conical shrinker). This is depicted in Figure \ref{fig:pseudolocality-1} (the region where we obtain curvature estimates is shaded in blue). When translated to the rescaled flow, this annular region will grow exponentially. This initially seems like a problem, since the inner boundary is also moving away exponentially. However, as long as the core graphical hypothesis is satisfied, we can use the pseudolocality estimates at a later time to get curvature estimates further inside. This is shown in Figure \ref{fig:pseudolocality-2}. The argument we have just described shows that as long as the core graphical hypothesis $(*_{b,\underline r})$ applies, we have that $\tilde \br_{\ell}(M_{\tau}) \geq C e^{\frac \tau 2}$ (see Lemma \ref{lemm:rough-improves}). 

Finally, we must show that the core graphical hypothesis $(*_{b,\underline r})$ together with the estimate we have just obtained on the rough conical scale $\tilde \br_{\ell}(M_{\tau})$ imply that the conical scale (i.e., the scale at which we can cut off \eqref{eq:intro-Loj-entire}) is comparable to the shrinker scale $\bR(M_{\tau})$. Since the rough conical scale is improving exponentially, it basically suffices to show that the conical and shrinker scales are comparable, when the shrinker scale is much smaller than the rough conical scale, i.e., $\bR(M_{\tau}) \ll \tilde \br_{\ell}(M_{\tau})$ (see \eqref{eq:huge-shrinker-scale-trivial-est} for the case where this does not hold). 

At this point, we can use the argument of Colding--Minicozzi from \cite[Corollary 1.28]{CM:uniqueness} to argue that because $\bR(M_{\tau}) \ll \tilde \br_{\ell}(M_{\tau})$, the function $\phi_{M_{\tau}} = \frac 12 \bangle{x,\nu_{M_{\tau}}} - H_{M_{\tau}}$ (which measures how close $M_{\tau}$ is to a shrinker) must be very small (see the proof of Theorem \ref{theo:final-Loj}). 

Finally, we show that this (along with the rough conical scale $\tilde \br_{\ell}(M_{\tau})$ estimates) suffices to extend the graphicality (and decay estiamates) from the core $B_{\underline r}$ nearly all the way out to the shrinker scale $\bR(M_{\tau})$ (see Proposition \ref{prop:approx-up-to-rough}). Because this step is delicate and forms a key part of the argument, we explain this argument in a model situation below.

\begin{lemm}[Model problem for the extension of the conical scale]\label{lemm:model-problem}
Fix $\beta_{0}>0$ and suppose that $u: \RR^{2}\to\RR$ satisfies 
\[
\cL_{\frac 12 } u : = \Delta u - \frac 12 (r \partial_{r} u - u) = 0
\]
on $\RR^{2}$ and $|\nabla^{k}u| = O(r^{1-k})$ for all $k\in\NN$. Finally, assume that $\Vert u \Vert_{C^{3}(B_{\underline r+2})} \leq b$ for $\underline r$ sufficiently large and $b$ sufficiently small depending\footnote{We will think of the $|\nabla^ku|$ estimates as being given a priori, so everything here is allowed to depend on the implied constants.} on $\beta_{0}$. Then, there is $c:\SS^{1}\to\RR$ and $f:\RR^{2}\to\RR$ so that outside of $B_{1}$ 
\[
u(r,\theta) = c(\theta) r + f(r,\theta)
\]
and $\Vert c\Vert_{C^{0}(\SS^{1})} + \Vert r f \Vert_{C^{0}(\RR^{2})} \leq \beta_{0}$
\end{lemm}

Before proving this lemma, we explain the relationship with the full improvement/extension result (Proposition \ref{prop:approx-up-to-rough}). Firstly, we have considered the simplest possible conical shrinker $\RR^{2}\subset \RR^{3}$ instead of a general asymptotically conical shrinker $\Sigma^{n}\subset \RR^{n+1}$. In the full problem, we have that $\phi_{M_{\tau}}$ is very small, so the part of $M_{\tau}$ that is graphical over $\Sigma$ roughly solves the graphical shrinker equation. The $\cL_{\frac 12}$ operator is the linearization (at the flat plane) of the shrinker equation, so to simplify this situation we have simply assumed that $\cL_{\frac 12}u = 0$. The higher derivative estimates on $u$ are the analogue here of the rough conical scale estimates. Finally, the $C^{3}$-smallness of $u$ in $B_{\underline r+2}$ is analogous to the core graphical hypothesis. We have simplified the conclusion above, in Proposition \ref{prop:approx-up-to-rough} we prove full $\cC\cS^{2,\alpha}_{-1}(\Sigma)$ estimates for $u$ (but the result described here contains the essential ideas).

We note that a key technical difficulty present in Proposition \ref{prop:approx-up-to-rough} that does not occur in this model problem is the fact that $M_{\tau}$ is \emph{not} an entire graph over $\Sigma$ (and a priori is only graphical up to $B_{\underline r}$). Thus, the argument below must be coupled with a continuity argument outwards; this necessarily complicates the argument. 

\begin{proof}
The beginning of the proof is very similar to proof of \cite[Theorem 8.9]{KKM:shrinkers}. As an initial step, we treat the Laplacian in $\cL_{\frac 12}$ as an error term, since $\Delta u = O(r^{-1})$ from the Hessian estimates on $u$. Thus, we find that 
\begin{equation}\label{eq:u-ode-model-problem}
r^{2} \partial_{r} \left( \frac u r \right) = r \partial_{r} u - u = O(r^{-1}).
\end{equation}
Integrating this to infinity, we find 
\[
c(\theta) : = \lim_{r\to\infty} \frac{u(r,\theta)}{r}
\]
is well defined (and continuous). Thus, we have obtained the asserted decomposition. It remains to prove the asserted estimates for $c$ and $f$.

We begin by proving that $\frac{u}{r}$ is small (we have already proven that it is bounded). Integrating \eqref{eq:u-ode-model-problem} from $\underline r$ to $r$, we find that 
\begin{equation}\label{eq:u-model-prob-int-first}
\frac{u(r,\theta)}{r} - \frac{u(\underline r,\theta)}{\underline r} = O(\underline r^{-2} - r^{-2}). 
\end{equation}
In particular,
\[
c(\theta) = \frac{u(\underline r,\theta)}{\underline r} + O(\underline r^{-2}). 
\]
We can arrange that the right hand side is less than $\frac {\beta_{0}}{2}$ by choosing $\underline r$ large (to control the second term) and $b$ small (to control the first term). This proves the desired estimate for $c(\theta)$. 

We now turn to the estimate for $f$. The key idea is to interpolate smallness in the $C^{0}$ norm of $u$ (that we have just obtained) with scale invariant boundedness of higher derivatives: this implies that the Laplacian term in $\cL_{\frac 12}$ is controlled with a small constant. Then, integrating the resulting ODE estimate to infinity, we obtain decay (and, more importantly,\footnote{Note that the initial step in the proof can be used to prove decay for $f$, but not smallness.} smallness) estimates for $f$. 

First of all, we note that by \eqref{eq:u-model-prob-int-first}, we have
\[
\left| u(r,\theta) \right| \leq \delta^{2}r,
\]
for $r\geq \underline r$, where we can take $\delta$ small below (at the cost of taking $\underline r$ larger and $b$ smaller). Interpolating this (on balls of unit size) with $|D^{k}u| = O(r^{1-k})$, for $k$ large, we find that
\[
|\Delta u| \leq O(\delta) r^{-1},
\]
for $r \geq \underline r$. Now, returning to $\cL_{\frac 12}u = 0$ we have gained smallness in the constant on the right hand side of \eqref{eq:u-ode-model-problem}, i.e.,
\[
\partial_{r} \left( \frac u r \right) = O(\delta) r^{-3} .
\]
Now, integrating this on $[r,\infty)$, we find
\[
c(\theta) = \frac{u(r,\theta)}{r} + O(\delta r^{-2}). 
\]
Because $u(r,\theta) = c(\theta)r + f(r,\theta)$, this gives
\[
f(r,\theta) = O(\delta r^{-1}).
\]
Choosing $\delta$ sufficiently small (in terms of $\beta_{0})$, we find that $\Vert r f \Vert_{C^{0}(\RR^{2}\setminus B_{\underline r}(0))}\leq \frac{\beta_{0}}{2}$. This completes the proof (since we already control $u$, and thus $f$ inside of $B_{\underline r}(0)$). 
\end{proof}

At this point, we have proven that the conical scale $\br_{\ell}(M_{\tau})$ is sufficiently large, so that when cutting off the \L ojaisewicz--Simon inequality \eqref{eq:intro-Loj-entire} at this scale, the error terms do not affect the right hand side of the equation. At this point, we can use the now-standard uniqueness argument based on the \L ojasiewicz inequality for parabolic equations (cf.\ {\cite{Schulze:compact, Simon:Loj}}). This completes the sketch of the proof of Theorem \ref{thm:unique}. 

\subsection{Organization of the paper} In Section \ref{sec:prelim} we prove several estimates on the geometry of asymptotically conical self-shrinkers. In Section \ref{sec:lin-est-spaces} we establish the relevant linear PDE theory in weighted H\"older and Sobolev spaces. In Section \ref{sec:Loj-entire}, we apply these estimates to establish the \L ojasiewicz--Simon inequality for entire graphs over a conical shrinker. So as to localize this inequality, in Section \ref{sec:scales} we define the various scales used later. This then allows us to localize the inequality in Section \ref{sec:loc-LS}. In Section \ref{sec:improvement-argument} we carry out the central improvement/extension argument (cf.\ the model problem Lemma \ref{lemm:model-problem} above). In Section \ref{sec:final-Loj}, we establish our final \L ojasiewicz--Simon inequality. Putting this all together, we prove the uniqueness of conical tangent flows (Threorem \ref{thm:unique}) in Section \ref{sec:unique}. In Appendix \ref{app:defi}, we recall several standard definitions and conventions, while in Appendix \ref{app:interpolate} we recall some useful interpolation inequalities. Appendix \ref{app:NG} contains an analysis of normal graphs and Appendix \ref{app:EL} recalls the first and second variations of Gaussian area. Appendix \ref{app:area-bds} recalls an entropy-area bound estimate. Finally, we include a list of notation.

\subsection{Acknowledgements}

We are grateful to Jacob Bernstein for several useful discussions.  O.C.~was partially supported by a Sloan Fellowship, a Terman Fellowship, and NSF grants  DMS-1811059 and DMS-2016403. F.S. was supported by a Leverhulme Trust Research Project Grant RPG-2016-174.

\section{Geometric preliminaries} \label{sec:prelim}

Throughout this section we fix $\Sigma^{n}\subset \RR^{n+1}$ a smooth, smoothly asymptotically conical self-shrinker. We denote by
\[
\cC = \lim_{t\nearrow0}\sqrt{-t}\Sigma
\]
the asymptotic cone of $\Sigma$ and assume that $\cC^{n}$ is the cone over $\Gamma^{n-1}\subset \SS^{n}$. Note that the induced metric on $\cC$ satisfies
\[
g_{\cC} = dr\otimes dr + r^{2}g_{\Gamma}
\]
for $r=|x|$ the radial variable. 

The following estimate is a straightforward consequence of the smooth convergence of $\sqrt{-t}\Sigma$ to $\cC$ combined with scaling considerations.
\begin{lemm}
For $R>0$ sufficiently large, the induced metric, $g_{\Sigma}$, on $\Sigma\setminus B_{R}(0)$ satisfies
\[
g_{\Sigma} = g_{\cC} + h 
\]
for $h$ a symmetric $(0,2)$-tensor on $\Sigma\setminus B_{R}(0)$ satisfying $|\nabla^{(j)}h| = o(r^{-j})$ as $r\to\infty$, for all $j\geq 0$. The second fundamental form of $\Sigma$ satisfies 
\[
|\nabla^{(j)}A_{\Sigma}| = O(r^{-j-1})
\]
as $r\to\infty$ for $j\geq 0$. 
\end{lemm}
In the sequel, we will improve these estimates based on the fact that $\Sigma$ is a self-shrinker. Indeed, the shrinker equation \eqref{eq:self-shrinker} and second fundamental form decay in the previous lemma combine to yield decay for $\bangle{x,\nu_{\Sigma}}$ that is faster than scaling:
\begin{coro}\label{coro:improved-x-perp}
For $R>0$ sufficiently large, we have
\[
|\nabla^{(j)}\bangle{x,\nu_{\Sigma}}| = O(r^{-j-1})
\]
as $r\to\infty$ for $j\geq 0$. 
\end{coro}

\subsection{Improved conical estimates for shrinkers}\label{subsec:improved-conical-est}

\begin{lemm}\label{lemm:improved-conical-est-shrinker}
For $R>0$ sufficiently large, there is $w \in C^{\infty}(\cC\setminus B_{R}(0))$ so that 
\[
\Graph w : = \{ p + w(p)\nu_{\cC}(p) : p \in \cC \setminus B_{R}(0)\} \subset \Sigma
\]
parametrizes $\Sigma$ outside of a compact set. The function $w$ satisfies 
\[
w = O(r^{-1})
\]
and 
\[
\nabla^{(j)} w = O(r^{-1-j+\eta})
\]
as $r\to\infty$ for any $\eta>0$ and $j\geq 1$. Moreover, the radial derivatives satisfy the sharper relation $\partial^{(j)}_{r} w = O(r^{-1-j})$. 
\end{lemm}
\begin{proof}
For $p \in \Gamma$, consider the plane $T_{p}\cC$ with normal vector $\nu_{\cC}(p)$. After a rotation, we can assume that $T_{p}\cC = \{x^{n+1}=0\}$ and $\nu_{\cC}(p) = \pm e^{n+1}$. Define
\[
\Gamma_{\eps,R} : = \{ x \in T_{p}\cC : |\langle x,p\rangle|  >(1-\eps)|x|, |x|> R\}.
\]
For $\eps>0$ sufficiently small and $R$ sufficiently large, there is $u,u_{\infty} : \Gamma_{\eps,R}\to\RR$ so that
\begin{align*}
\Graph u & = \{(y,u(y)) : y \in \Gamma_{\eps,R}\} \subset \Sigma,\\
\Graph u_{\infty} & = \{(y,u_{\infty}(y)) : y \in \Gamma_{\eps,R}\} \subset \cC.
\end{align*}
We have that
\[
\nabla^{(j)} u(y) = \nabla^{(j)}u_{\infty}(y) + o(|y|^{1-j})
\]
as $y\to\infty$. 

We recall that
\[
\nu_{\Sigma} = \pm \frac{(-\nabla u,1)}{\sqrt{1+|\nabla u|^{2}}} 
\]
Thus, by Corollary \ref{coro:improved-x-perp}, we find that 
\begin{equation}\label{eq:equation-u-shrinker-radial-derivative}
\bangle{y,\nabla u(y)} - u(y) = O(|y|^{-1}). 
\end{equation}
Thus, the function $v(s) = \frac{u(sp)}{s}$ satisfies $\lim_{s\to\infty} v(s) = 0$ (because $u_{\infty}(sp) = 0$) and $v'(s) = O(s^{-3})$ by \eqref{eq:equation-u-shrinker-radial-derivative}. Integrating this, we find that
\begin{equation}\label{eq:improved-0th-order-decay-conical-shrinker}
u(sp) = O(s^{-1}). 
\end{equation}
Thus (taking $R$ larger if necessary), we may find $w \in C^{\infty}(\Sigma\setminus B_{R}(0))$ so that
\[
\Graph w : = \{q + w(q) \nu_{\cC}(q) : q \in \cC\setminus B_{R}(0)\}\subset \Sigma
\]
parametrizes $\Sigma$ outside of a compact set. From \eqref{eq:improved-0th-order-decay-conical-shrinker} we find that
\[
|w| = O(r^{-1}).
\]
This yields the first asserted decay estimate. Furthermore, scaling considerations yield
\[
|\nabla^{(j)} w| = o(r^{1-j}),
\]
as $r\to\infty$ for $j\geq 1$. Hence, the second assertion follows by interpolating these two estimates (cf.\ Lemma \ref{lemm:interpolation}). Finally, by differentiating \eqref{eq:equation-u-shrinker-radial-derivative} in the radial direction, the improved radial derivative estimate follows. 
\end{proof}

\begin{coro}\label{coro:metric-conical-shrinker}
For $R>0$ sufficiently large, we have the following improved estimates on the induced metric: 
\[
g_{\Sigma} = dr\otimes dr + r^{2}g_{\Gamma} + h 
\]
for $h$ a symmetric $(0,2)$-tensor on $\Sigma\setminus B_{R}(0)$ satisfying $|h| = O(r^{-2})$ and $|\nabla^{(j)}h| = O(r^{-2-j+\eta})$ as $r\to\infty$, for all $j\geq 1$ and $\eta>0$. 
\end{coro}
\begin{proof}
Write $F: \cC\setminus B_{R}(0)\to\Sigma$, $F(p) = p+w(p) \nu_{\cC}(p)$. We compute (using the fact that $A_{\cC}(\partial_{r},\cdot) = 0$)
\begin{align*}
\partial_{r}F & = \partial_{r} + (\partial_{r}w(p)) \nu_{\cC}(p),\\ 
r^{-1}\partial_{\omega_{i}} F & = r^{-1} \partial_{\omega_{i}} + r^{-1}(\partial_{\omega_{i}} w(p)) \nu_{\cC}(p) - w(p) A_{\cC}|_{p}(r^{-1}\partial_{\omega_{i}},\cdot) .
\end{align*}
That $|h| = O(r^{-2})$ follows from these expressions and Lemma \ref{lemm:improved-conical-est-shrinker}. The higher derivative estimates follow from interpolation, as in Lemma \ref{lemm:improved-conical-est-shrinker}. 
\end{proof}

\begin{lemm}\label{lemm:normal-Sigma-estimates}
The unit normal to $\Sigma$ satisfies
\begin{align*}
\nu_{\Sigma}(F(p))
& =O(r^{-2})\partial_{r} + \sum_{j=1}^{n-1} O(r^{-2+\eta}) r^{-1}\partial_{\omega_{j}} + (1-O(r^{-4+\eta})) \nu_{\cC}(p)
\end{align*}
for $\eta>0$ as $r\to\infty$.  
\end{lemm}
\begin{proof}
Write 
\begin{equation}\label{eq:unit-normal-sigma-components-ABC}
\nu_{\Sigma}(F(p)) = A \partial_{r} + \sum_{j=1}^{n-1}B_{j} r^{-1}\partial_{\omega_{j}} + C \nu_{\cC}(p),
\end{equation}
where
\[
A^{2} + \sum_{j=1}^{n-1} B_{j}^{2} + C^{2} = 1.
\]
Because $\bangle{\nu_{\Sigma},\partial_{r}F} = \bangle{\nu_{\Sigma},r^{-1}\partial_{\omega_{i}}F} = 0$, we find that
\begin{align*}
0 & = A + C (\partial_{r}w(p)) \\
0 & =  \sum_{j=1}^{n-1} B_{j} (\delta_{ij}+O(r^{-2})) + C(r^{-1}\partial_{\omega_{i}}w(p))
\end{align*}
This implies the claim. 
\end{proof}

\begin{lemm}\label{lemm:cancellation-w-derivatives}
We have $|\nabla^{(j)}(r\partial_{r}w(p) - w(p))| = O(r^{-1-j})$ for any $j\geq 0$. 
\end{lemm}
\begin{proof}
Revisiting the proof of Lemma \ref{lemm:normal-Sigma-estimates}, we find that the components of $\nu_{\Sigma}$ in \eqref{eq:unit-normal-sigma-components-ABC} satisfy 
\begin{align*}
A & = -C (\partial_{r}w(p)) \\
0 & =  \sum_{j=1}^{n-1} B_{j} (\delta_{ij}+\tilde b_{j}) + C(r^{-1}\partial_{\omega_{i}}w(p))
\end{align*}
where $|\tilde b_{j}| = O(r^{-2})$ and $|\nabla^{(j)}\tilde b_{j}| = O(r^{-2-j+\eta})$. Thus, we find that the expressions from the proof of Lemma \ref{lemm:normal-Sigma-estimates} can be differentiated in the sense that
\begin{align*}
A & = - \partial_{r}w(p) + a\\
B_{i} & = - r^{-1}\partial_{\omega_{i}}w(p) + b_{i}\\
C & = 1 + c,
\end{align*}
where $|\nabla^{(j)}a|=O(r^{-5 - j+\eta})$, $|\nabla^{(j)}b_{i}| = |\nabla^{(j)}c| = O(r^{-4-j+\eta}))$. This implies that 
\begin{align*}
\bangle{F(p),\nu_{\Sigma}(F(p))} & =  r A + w(p) C\\
& = w(p) - r\partial_{r}w(p) + (ar - cw).
\end{align*}
Using Corollary \ref{coro:improved-x-perp} and the above estimates for $a,c$, we conclude the proof.
\end{proof}

\begin{lemm}\label{lemm:second-fund-form}
The second fundamental form of $\Sigma$ satisfies
\begin{align*}
A_{\Sigma}(\partial_{r}F,\partial_{r}F) & = O(r^{-3})\\
A_{\Sigma}(\partial_{r}F,r^{-1}\partial_{\omega_{i}}F) & = O(r^{-3})\\
A_{\Sigma}(r^{-1}\partial_{\omega_{i}}F,r^{-1}\partial_{\omega_{j}}F) & = A_{\cC}(r^{-1}\partial_{\omega_{i}},r^{-1}\partial_{\omega_{j}}) +O(r^{-3+\eta})
\end{align*}
as $r\to\infty$. Moreover, $|\nabla^{(k)}_{\cC} (A_{\Sigma}\circ F - A_{\cC})| = O(r^{-3-k+\eta})$ for any $\eta>0$ and $k\geq 1$. 
\end{lemm}

\begin{proof}
We compute
\begin{align*}
\partial^{2}_{r,r}F & = (\partial^{2}_{r,r}w(p)) \nu_{\cC}(p)\\
r^{-1}\partial^{2}_{r,\omega_{i}} F & = r^{-2}\partial_{\omega_{i}} + r^{-1}(\partial^{2}_{r,\omega_{i}}w)\nu_{\cC}(p) - (\partial_{r}w(p))A_{\cC}|_{p}(r^{-1}\partial_{\omega_{i}},\cdot)\\
r^{-2} \partial^{2}_{\omega_{i},\omega_{j}}F & = A_{\cC}|_{p}(r^{-1}\partial_{\omega_{i}},r^{-1}\partial_{\omega_{j}}) \nu_{\cC}(p) + r^{-2} (\partial^{2}_{\omega_{i},\omega_{j}}w(p)) \nu_{\cC}(p)\\
& \qquad  - r^{-1}(\partial_{\omega_{i}}w(p)) A_{\cC}|_{p}(r^{-1}\partial_{\omega_{j}},\cdot)  - r^{-1}(\partial_{\omega_{j}}w(p)) A_{\cC}|_{p}(r^{-1}\partial_{\omega_{i}},\cdot) \\
& \qquad - w(p) (\nabla_{r^{-1}\partial_{\omega_{j}}} A_{\cC})|_{p}(r^{-1}\partial_{\omega_{i}},\cdot). 
\end{align*}
Using Lemma \ref{lemm:normal-Sigma-estimates}, the first and third equation follow immediately. 
For the second, we use the expression for $r^{-1}\partial_{w_{i}}F$ (which is orthogonal to $\nu_{\Sigma}(F(p))$ to write 
\begin{align*}
r^{-1} \partial^{2}_{r,\omega_{i}} F = r^{-2}\partial_{\omega_{i}}F + r^{-2}(r\partial^{2}_{r,\omega_{i}}w - \partial_{\omega_{i}}w(p))\nu_{\cC}(p) + r^{-1}(w(p) - r\partial_{r}w(p))A_{\cC}|_{p}(r^{-1}\partial_{\omega_{i}},\cdot)
\end{align*}
Using Lemmas \ref{lemm:normal-Sigma-estimates} and \ref{lemm:cancellation-w-derivatives}, the first estimates follow. The higher derivatives follow by differentiating these expressions. 
\end{proof}

\begin{lemm}\label{lemm:proj-radial-vect-Sigma}
The vector field $V : = \proj_{T\Sigma}F(p) - r\partial_{r} F$ is tangent to $\Sigma$ and satisfies $|V| = O(r^{-1})$, $|\nabla^{(k)}V| = O(r^{-1-k+\eta})$ for $\eta>0$. 
\end{lemm}
\begin{proof}
Because $\bangle{F(p),\nu_{\Sigma}(F(p))} = O(r^{-1})$, we compute
\begin{align*}
\proj_{T\Sigma} F(p) & = F(p) - \bangle{F(p),\nu_{\Sigma}(F(p))}\nu_{\Sigma}(F(p))\\
& = p + w(p)\nu_{\cC}(p)+ O(r^{-1}) \\
& = r\partial_{r} + w(p)\nu_{\cC}(p) + O(r^{-1}) \\
& = r\partial_{r}F +  O(r^{-1}).
\end{align*}
The higher derivatives follow similarly. 
\end{proof}

The function $w$ from Lemma \ref{lemm:improved-conical-est-shrinker} gives a diffeomorphism from $\cC \setminus B_{R}(0) \simeq \Gamma \times [R,\infty)$ to the non-compact part of $\Sigma$, where we recall that $\Gamma$ is the link of the asymptotic cone $\cC$. We will thus parametrize points of $\Sigma$ by $(r,\omega) \in \Gamma \times [R,\infty)$ below. We will write $g_{\cC}$ for the metric on the end of $\Sigma$ given by
\[
g_{\cC} = dr\otimes dr + r^{2}g_{\Gamma}
\]
in this parametrization. We emphasize that the coordinate $r$ along $\Sigma$ is \emph{not} exactly equal to $d_{\RR^{3}}(\cdot,0)$ (like it is along the cone).  It is useful to extend $r$ to $\tilde r$ defined on all of $\Sigma$ so that $\tilde r\geq 1$ on $\Sigma$ and $\tilde r=r$ outside of $B_{R}$ for $R$ as above.

\begin{lemm}\label{lemm:radial-der-improved}
The radial derivative satisfies
\[
\vec{x}\cdot \nabla_{\Sigma} f =r\partial_{r}f + \alpha_{3}\cdot\nabla_{g_{\cC}} f,
\]
where $|\alpha_{3}| = O(r^{-1})$ and $|\nabla^{(j)}\alpha_{3}| = O(r^{-1-j+\eta})$ for $\eta>0$ and $j\geq 1$, as $r\to\infty$. 
\end{lemm}
\begin{proof}
This follows from Lemma \ref{lemm:proj-radial-vect-Sigma}. 
\end{proof}

\section{Linear estimates in weighted spaces}\label{sec:lin-est-spaces}

In this section we consider the relevant weighted function spaces which will play a role in our proof of the \L ojasiewicz--Simon inequality for the conical shrinker $\Sigma^{n}\subset \RR^{n+1}$. Our choice of H\"older spaces will be heavily influenced by the work of N.\ Kapouleas, S.~J.\ Kleene, and N.~M.\ M\o ller \cite{KKM:shrinkers} except for the complication that in \cite{KKM:shrinkers}, it was only neccessary to define the spaces on a flat $\RR^{2}$ (which is, of course, a conical shrinker), whereas, here we must consider general conical shrinkers. Additionally, in various points of \cite{KKM:shrinkers}, the discrete symmetry of the problem was used in certain places, which will not be available to us here.

\subsection{Weighted H\"older spaces} We now define the relevant weighted H\"older spaces. We begin with the most basic weighted space. 
\begin{defi}[Homogeneously weighted H\"older spaces]
We define a norm, for $\gamma\in\RR$,
\[
\Vert f \Vert_{0;-\gamma}^{\text{hom}} : = \sup_{x\in\Sigma} \tilde r(x)^{\gamma} |f(x)|
\]
and a semi-norm
\[
[ f ]_{\alpha;-\gamma-\alpha}^{\text{hom}} : = \sup_{x,y\in \Sigma} \frac{1}{\tilde r(x)^{-\gamma-\alpha} + \tilde r(y)^{-\gamma-\alpha}} \frac{|f(x)-f(y)|}{|x-y|^{\alpha}}.
\]
We thus define $C^{0,\alpha}_{\text{hom},-\gamma}(\Sigma)$\index{$C^{0,\alpha}_{\text{hom},-\gamma}(\Sigma)$} to be the set of functions $f:\Sigma\to\RR$ so that
\[
\Vert f \Vert_{0,\alpha;-\gamma}^{\text{hom}} : = \Vert f \Vert_{0;-\gamma}^{\text{hom}} + [ f ]_{\alpha;-\gamma-\alpha}^{\text{hom}} 
\]
is finite. Similarly, we define $C^{2,\alpha}_{\text{hom},-\gamma}(\Sigma)$\index{$C^{2,\alpha}_{\text{hom},-\gamma}(\Sigma)$} to be the set of $f:\Sigma\to\infty$ so that the norm
\[
\Vert f \Vert_{2,\alpha;-\gamma}^{\text{hom}} = \sum_{j=0}^{2} \Vert (\nabla_{\Sigma})^{(j)} f \Vert_{0,\alpha;-\gamma}^{\text{hom}}
\]
is finite.
\end{defi}
Loosely speaking, $C^{2,\alpha}_{\text{hom},-\gamma}(\Sigma)$ is the space of $C^{2,\alpha}$ functions whose $C^{2,\alpha}$ norm falls off like $r^{-\gamma}$ at infinity. We now define a space which will require stronger weights in the radial direction.

\begin{defi}[Anisotropically weighted H\"older spaces]
We define $C^{2,\alpha}_{\text{an},-1}(\Sigma)$\index{$C^{2,\alpha}_{\text{an},-1}(\Sigma)$} to be the space of $f\in C^{2,\alpha}_{\text{hom},-1}(\Sigma)$ so that 
\[
\Vert f \Vert_{2,\alpha;-1}^{\text{an}} : = \Vert f \Vert_{2,\alpha;-1}^{\text{hom}} + \Vert \vec{x}\cdot \nabla_{\Sigma} f \Vert_{0,\alpha;-1}^{\text{hom}}
\]
is finite. 
\end{defi}

Now, we fix a cutoff function $\chi : [0,\infty) \to [0,1]$ so that $\supp \chi \subset [R,\infty)$, $\chi \equiv 1$ in $[2R,\infty)$, and $|\nabla^{j} \chi| \leq C R^{-j}$ for $j\geq 1$ and $C$ independent of $R$ sufficiently large. This now allows us to define our primary H\"older space. 
\begin{defi}[Cone H\"older spaces]
We define \index{$\cC\cS^{0,\alpha}_{-1}$}$\cC \cS^{0,\alpha}_{-1}(\Sigma) : = C^{0,\alpha}_{\textrm{hom},-1}(\Sigma)$ and \index{$\cC\cS^{2,\alpha}_{-1}$}
\[
\cC\cS^{2,\alpha}_{-1}(\Sigma) : = C^{2,\alpha}(\Gamma) \times C^{2,\alpha}_{\text{an},-1}(\Sigma).
\]
An element $(c,f) \in \cC\cS_{-1}^{2,\alpha}(\Sigma)$ will be considered as a function on $\Sigma$ given by
\[
u=u_{(c,f)}(r,\omega) = \chi(r) c(\omega) r + f(r,\omega)
\]
for $r\geq R$, and $u=f$ otherwise. We will frequently conflate $u$ with $(c,f)$. We take the norm
\[
\Vert u \Vert_{\cC\cS_{-1}^{2,\alpha}(\Sigma)} : = \Vert c\Vert_{C^{2,\alpha}(\Gamma)} + \Vert f\Vert_{2,\alpha;-1}^{\text{an}}.
\]
\end{defi}
Observe that an element of $\cC\cS_{-1}^{2,\alpha}(\Sigma)$ is allowed to grow linearly at infinity, but only in a particularly prescribed manner. The remaining terms then must decay like $r^{-1}$. It is a standard exercise to observe that all of the above spaces are indeed Banach spaces. 

\subsection{Mapping properties}We observe that the cone spaces are well suited to the analysis of the $\cL_{\frac 12}$ operator 
\[
\cL_{\frac12} u  : = \Delta_{\Sigma} u - \frac 12 \vec{x}\cdot \nabla_{\Sigma} u + \frac 12 u
\]
(see also Definition \ref{defi:L-operators}) in the following sense. 
\begin{lemm}
For $a: \Sigma\to\RR$ with $\Vert a \Vert_{C^{0,\alpha}(B_{1}(x))} = O(|x|^{-2})$ for $x\in\Sigma$ with $|x|\to\infty$, i.e., $a \in C^{0,\alpha}_{\textnormal{hom};-2}(\Sigma)$, we have that the operator
\[
\cL_{\frac 12} + a : \cC\cS^{2,\alpha}_{-1}(\Sigma) \to \cC\cS^{0,\alpha}_{-1}(\Sigma)
\]
is bounded.
\end{lemm}
\begin{proof}
This follows directly from the definition of the cone spaces (after observing that the linear term $r c(\omega)$ exhibits a cancelation in the term $\frac 12 (u - \vec{x}\cdot\nabla^{\Sigma} u)$; note that this fact does not hold for general $\cL_{\gamma}$ when $\gamma\not = \frac 12$). 
\end{proof}

\subsection{Schauder estimates} In this section, we prove Schauder estimates for the $\cL$ operator in the cone H\"older spaces. These estimates are essentially the generalization of \cite[Proposition 8.8]{KKM:shrinkers} to our setting, and we will closely follow their arguments, with some necessary modifications as discussed above. We note that Schauder estimates for the linearization of the expander equation on asymptotically conical self-expanders were proven by a related method in \cite[Proposition 5.3]{BernsteinWang:space-of-expanders}.

\begin{prop}\label{prop:CS-schauder}
Consider $a:\Sigma\to\RR$ with $\Vert a \Vert_{C^{0,\alpha}(B_{1}(x))} = O(|x|^{-2})$ for $x\in\Sigma$ with $|x|\to\infty$, i.e., $a \in C^{0,\alpha}_{\textnormal{hom};-2}(\Sigma)$. Then, there is $C=C(\Sigma,a)$ so that if $u \in C^{2,\alpha}_{\textnormal{loc}}(\Sigma) \cap C^{0}_{\textnormal{hom};+1}(\Sigma)$ has $\cL_{\frac 12} u + au  \in \cC\cS^{0,\alpha}_{-1}(\Sigma)$, then $u\in \cC\cS^{2,\alpha}_{-1}(\Sigma)$ and we have the estimate
\[
\Vert u \Vert_{\cC\cS_{-1}^{2,\alpha}(\Sigma)} \leq C\left( \Vert u \Vert_{C^{0}_{\textnormal{hom},+1}(\Sigma)} + \Vert \cL_{\frac 12} u + au  \Vert_{\cC\cS_{-1}^{0,\alpha}(\Sigma)} \right).
\] 
\end{prop}

Because the $\cL$ operator is related to the linearization of the shrinker equation, which is, in turn, a special case of the mean curvature flow (whose linearization is related to the heat equation), we might expect that such an estimate can be proven from standard parabolic Schauder estimates. This is nearly the case, except it turns out the appropriate time parametrization of the equations will produce functions which are not H\"older continuous (at $t=0$) in the time variables. As such, we will require the following non-standard parabolic Schauder estimates due to A.\ Brandt \cite{Brandt:interiorSchauder}. We note that these estimates were strengthened in \cite{Knerr:interiorSchauder} (see also \cite{Lieberman:IV}) but we will not make use of these stronger estimates here. 

\begin{theo}[Non-standard interior Schauder estimates, {\cite{Brandt:interiorSchauder}}]\label{theo:brandt-schauder} 
Suppose that $B_{2}\subset \RR^{n}$ and we are given coefficients $a_{ij}(x,t),b_{i}(x,t),c:B_{2}\times [-2,0]\to\RR$ and functions $f,u : B_{2} \times [-2,0]\to\RR$ so that $u$ is a classical solution of
\[
\frac{\partial u}{\partial t} - a_{ij}D^{2}_{ij}u - b_{i}D_{i}u - cu = f.
\]
Assume that the coefficients $a_{ij},b_{i},c$ have spatial H\"older norms bounded uniformly in time, e.g.,
\[
\sup_{t \in [-2,0]} \left( \Vert a_{ij}(\cdot,t)\Vert_{C^{0,\alpha}(B_{1})} + \Vert b_{i}(\cdot,t)\Vert_{C^{0,\alpha}(B_{1})} + \Vert c(\cdot,t)\Vert_{C^{0,\alpha}(B_{1})}  \right) < \Lambda
\]
and that the equation is uniformly parabolic in the sense that
\[
a_{ij}(x,t) \xi_{i}\xi_{j} \geq \lambda |\xi|^{2}
\]
for $\lambda > 0$. Then, for $T \in (-1,0]$,
\[
\sup_{t\in[-1,T]} \Vert u(\cdot,t)\Vert_{C^{2,\alpha}(B_{1})} \leq C  \sup_{t\in[-2,T]}  \left( \Vert u(\cdot,t) \Vert_{C^{0}(B_{2})} +\Vert f(\cdot,t)\Vert_{C^{0,\alpha}(B_{2})} \right)
\]
for some $C=C(n,\lambda,\Lambda)$. 
\end{theo}

We now explain how to relate the $\cL_{\frac12}$-operator considered in Proposition \ref{prop:CS-schauder} to a parabolic equation where we can apply Theorem \ref{theo:brandt-schauder}. 
\begin{defi}[Intrinsic shrinker quantities] It is useful to consider the intrinsic behavior of the shrinker $\Sigma$ under the mean curvature flow. To this end, for $t\in [-1,0)$, we define the (time dependent) vector field $X_{t} = \frac{1}{2(-t)}x^{T}$. Here, $x^{T}$ is the tangential component of the position vector along $\Sigma$. For $t\in[-1,0)$, define $\Phi_{t}: \Sigma\to \Sigma$ to be the family of diffeomorphisms generated by $X_{t}$ (i.e., $\frac{\partial}{\partial t}\Phi_{t} = X_{t}\circ \Phi_{t}$) with $\Phi_{-1} = \Id$. Finally, define the metric $\hat g_{t} : = (-t) \Phi_{t}^{*} g_{\Sigma}$.
\end{defi}

Observe that if $F:\Sigma\to\RR^{n+1}$ is the embedding of $\Sigma$ in $\RR^{n+1}$, then 
\[
\hat F_{t} : = \sqrt{-t} (F \circ \Phi_{t}) : \Sigma \to \RR^{n+1}
\]
is a mean curvature flow of hypersurfaces parametrized by normal speed. Moreover, we have that $\hat g_{t} = \hat F_{t}^{*}g_{\RR^{n+1}}$. Thus, because the (extrinsic) blow-down of $\Sigma$ is $\cC$, we see that $(\Sigma,\hat g_{t},p)$ converges in the pointed $C^{\infty}$-Cheeger--Gromov sense to (the incomplete metric) $(\cC,g_{\cC},p)$ for any point $p$ sufficiently far out in the conical part of $\Sigma$. This will be useful in the sequel.

As in the proof of Corollary \ref{coro:metric-conical-shrinker} we write the end of $\Sigma$ via the map $F: \cC \setminus B_r(0) \rightarrow \Sigma, p \mapsto F(p) + w(p) \nu_\cC(p)$ as a normal graph over the cone $\cC$ with  coordinates $(r,\omega) \in \Gamma \times [R,\infty)$ for $R$ sufficiently large. We consider the induced flow of $\Phi_t$ in these coordinates, i.e.
$$ \tilde{\Phi}_t: = F^{-1}\circ \Phi_t \circ F\, $$ 
For $t \in [-1,0)$ we consider the map 
$$\phi_t: (R, \infty)\times \Gamma  \rightarrow (R, \infty)\times \Gamma, (r,\omega) \mapsto ((-t)^{-1/2} r, \omega)\, .$$
Then we have the following estimates.
\begin{lemm}\label{lemm:behavior-Phi-r}
For $t \in [-1,0)$, for $r$ sufficiently large, we have
\[
d_{g_\cC}\left(\tilde{\Phi}_t( r, \theta), \phi_t(r, \theta)\right) \lesssim \frac{1}{\sqrt{-t}\, r}.
\]
Moreover, in the coordinates\footnote{We emphasize that in this estimate we are \emph{not} using the conical metric, but rather the flat cylindrical metric $dr^{2} + g_{\Gamma}$ to estimate these derivatives. This avoids defining derivatives of diffeomorphisms as sections of an appropriate bundle and this estimate here suffices for our purposes.} $(r,\omega)$ we have the (non-sharp) estimate
\[
\big|D^{(j)}\big(\tilde{\Phi}_t - \phi_t\big)\big| (r,\theta)  \lesssim \frac{1}{\sqrt{-t}\, r^{1+j-\eta}} \, .
\]
for $j \geq 1$ and $\eta>0$. 
\end{lemm}
\begin{proof}
We denote the ambient radius by $\underbar{r}(x):= |x|$ and compute along $\Sigma$, using Lemma \ref{lemm:proj-radial-vect-Sigma},
\begin{align*}
\frac{\partial}{\partial t}( \underbar{r}\circ \Phi_{t}) & = (\nabla_{\dot \Phi_{t}} \underbar{r}) \circ\Phi_{t}\\
& = \frac{1}{2(-t)} \bangle{x^{T},\nabla \underbar{r}}_{g} \circ \Phi_{t}\\
& = \frac{1}{2(-t)}( r \circ \Phi_{t} + O((r\circ \Phi_{t})^{-1})).
\end{align*}
Integrating this, we see that
\begin{align}\label{eq:behavior-Phi-r.1}
\frac{1}{\sqrt{-t}} \left(\underbar{r}(x) - \frac{c}{\underbar{r}(x)}\right)  \leq \underbar{r}(\Phi_{t}(x)) \leq \frac{1}{\sqrt{-t}} \left(\underbar{r}(x) - \frac{c}{\underbar{r}(x)}\right)\, . 
\end{align}

Now, we have that
\[
x^{T} = r \partial_{r}F + O(r^{-1})
\]
by Lemma \ref{lemm:proj-radial-vect-Sigma}. 
This implies that 
\begin{equation}\label{eq:behavior-Phi-r.2}
\frac{\partial}{\partial t} \tilde{\Phi}_t = \frac{1}{2(-t)} \left( r \partial_r + O(r^{-1})\partial_{r} + O(r^{-2})\partial_{\omega_{i}}  \right)
\end{equation}
where the right hand side is evaluated at $\tilde{\Phi}_{t}(\cdot)$. 
Note that $\phi_t$ satisfies
\[
\frac{\partial}{\partial t} \phi_t = \frac{1}{2(-t)}\,  r \partial_r ,
\]
where the right hand side is evaluated at $\phi_{t}(\cdot)$. In combination with \eqref{eq:behavior-Phi-r.1}, this implies that
\begin{align*}
\frac{\partial}{\partial t} d_{g_\cC}(\tilde{\Phi}_t(\cdot),\phi_t(\cdot)) & \leq \frac{1}{2(-t)} \left( d_{g_\cC}(\tilde{\Phi}_t(\cdot) ,\phi_t(\cdot)) + \frac{c}{\underline r(\Phi_{t}(\cdot))}\right)\\
& \leq \frac{1}{2(-t)} \left( d_{g_\cC}(\tilde{\Phi}_t(\cdot) ,\phi_t(\cdot)) + \frac{c}{\underline r(\cdot)} \sqrt{-t} \right).
\end{align*}
Integrating this yields
\[
d_{g_\cC}(\tilde{\Phi}_t(\cdot),\phi_t(\cdot)) \leq \frac{c}{\sqrt{-t}\, \underline r(\cdot)}
\]
The derivative estimates follow similarly. 
\end{proof}

Now, assume that $\cL_{\frac 12} u + a u =E$ for some $u\in C^{2,\alpha}_{\textrm{loc}}(\Sigma)$ and $a : \Sigma\to\RR$ with $\Vert a \Vert_{C^{0,\alpha}(B_{1}(x))} = O(|x|^{-2})$ for $x\in\Sigma$ with $|x|\to\infty$. We define 
\[
\hat u(x,t) := \sqrt{-t} (u (\Phi_{t}(x)), \qquad \hat E(x,t) := \frac{1}{\sqrt{-t}} E (\Phi_{t}(x)), \qquad \hat a(x,t) =  \frac{1}{(-t)} a(\Phi_{t}(x))
\]
Then, we find that
\[
\Delta_{\hat g_{t}} \hat u = \frac{1}{\sqrt{-t}} (\Delta_{g} u)\circ \Phi_{t},
\]
since the Laplacian is diffeomorphism invariant, as well as
\begin{align*}
\frac{\partial \hat u} {\partial t}  & = \sqrt{-t} (\nabla_{X_{t}} u) \circ \Phi_{t} - \frac{1}{2\sqrt{-t}} u \circ \Phi_{t}\\
& = \frac{1}{2\sqrt{-t}}  \left( \vec{x}\cdot \nabla_{\Sigma} u \circ \Phi_{t} - u \circ \Phi_{t}\right).
\end{align*}
We thus find that
\begin{equation}\label{eq:parabolic-L}
\frac{\partial \hat u} {\partial t} - \Delta_{\hat g_{t}} \hat u - \hat a \hat u = \hat E.
\end{equation}

We now use this equation in conjunction with Theorem \ref{theo:brandt-schauder} to prove the desired Schauder estimates. Observe that Lemma \ref{lemm:behavior-Phi-r} and the presumed decay of $a$ shows that $\hat a(\cdot,t)$ is uniformly bounded in $C^{0,\alpha}$ on sufficiently far out balls of unit size, allowing us to apply Theorem \ref{theo:brandt-schauder}. 
\begin{proof}[Proof of Proposition \ref{prop:CS-schauder}] We can choose $R$ sufficiently large such that the normal evolution of $\Sigma_t:=\sqrt{t}\cdot \Sigma$ for $t\in [-2,0)$ is almost orthogonal to $x$ outside of $B_{R/4}$. Applying Theorem \ref{theo:brandt-schauder} to \eqref{eq:parabolic-L} we find that (where the implied constant is independent of $R$ sufficiently large)
\begin{align*}
 &\sup_{t\in[-1,0)} \Vert D^{2}_{x} \hat u(\cdot,t)\Vert_{C^{0}(\Sigma_t \cap (B_{R+2}(0)\setminus B_{R+1}(0)))} + 
  \sup_{t\in[-1,0)} \Vert D_{x} \hat u(\cdot,t)\Vert_{C^{0}(\Sigma_t \cap (B_{R+2}(0)\setminus B_{R+1}(0)))}\\
 &+\sup_{t\in[-1,0)}  [D^{2}_{x} \hat u(\cdot,t)]_{\alpha ;\, \Sigma_t \cap (B_{R+2}(0)\setminus B_{R+1}(0)))} 
\\ 
&\qquad\lesssim \sup_{t\in[-2,0)} \Vert \hat u(\cdot,t)\Vert_{C^{0}(\Sigma_t \cap (B_{R+3}(0)\setminus B_{R}(0)))} 
+  \sup_{t\in [-2,0)}  [ \hat E(\cdot,t)]_{\alpha;\, \Sigma_t \cap (B_{R+3}(0)\setminus B_{R}(0)))} .
\end{align*}
On the other hand, Lemma \ref{lemm:behavior-Phi-r} implies that for $R$ sufficiently large, we can estimate the H\"older norms of $\hat u$ in terms of weighted norms of $u$ as follows:
\begin{align*}
& \sup_{t\in[-1,0)} \Vert D^{2}_{x} \hat u(\cdot,t)\Vert_{C^{0}(\Sigma \cap (B_{R+2}(0)\setminus B_{R+1}(0)))}
+  \sup_{t\in[-1,0)} \Vert D_{x} \hat u(\cdot,t)\Vert_{C^{0}(\Sigma \cap (B_{R+2}(0)\setminus B_{R+1}(0)))} \\
& \qquad +  \sup_{t\in[-1,0)}  [D^{2}_{x} \hat u(\cdot,t)]_{\alpha ; \Sigma \cap (B_{R+2}(0)\setminus B_{R+1}(0))} \\
& \gtrsim \sup_{x \in \Sigma \setminus B_{R+1}(0)} r(x) |D^{2}u(x)| \\
& \qquad + \sup_{x,y \in \Sigma \setminus B_{R+1}(0)} \frac{1}{r(x)^{-1-\alpha}+r(y)^{-1-\alpha}} \frac{| D^{2}u(x) - D^{2}u(y)|}{|x-y|^{\alpha}} .
\end{align*}
Arguing similarly for the other terms, we thus rewrite the above parabolic Schauder estimates as weighted elliptic estimates.
\begin{align*}
\sup_{x \in \Sigma \setminus B_{R+1}(0)} r(x) |D^{2}u(x)| &+ \sup_{x,y \in \Sigma \setminus B_{R+1}(0)} \frac{1}{r(x)^{-1-\alpha}+r(y)^{-1-\alpha}} \frac{| D^{2}u(x) - D^{2}u(y)|}{|x-y|^{\alpha}} \\ \\ 
& \lesssim \sup_{x \in \Sigma \setminus B_{R/2}(0)}  r(x)^{-1} |u(x)| + \sup_{x \in \Sigma \setminus B_{R/2}(0)}  r(x) |E(x)| \\
& \qquad + \sup_{x,y \in \Sigma \setminus B_{R/2}(0)} \frac{1}{r(x)^{-1-\alpha}+r(y)^{-1-\alpha}} \frac{| E(x) - E(y)|}{|x-y|^{\alpha}} .
\end{align*}
This implies
\begin{align*}
\Vert D^{2}u \Vert_{ \cC \cS^{0,\alpha}_{-1}(\Sigma\setminus B_{R+1}(0))}&= \sup_{x \in \Sigma \setminus B_{R+1}(0)} r(x) |D^{2}u(x)| \\
& \qquad + \sup_{x,y \in \Sigma  \setminus B_{R+1}(0)} \frac{1}{r(x)^{-1-\alpha}+r(y)^{-1-\alpha}} \frac{| D^{2}u(x) - D^{2}u(y)|}{|x-y|^{\alpha}} \\ \\ 
& \lesssim \sup_{x \in \Sigma \setminus B_{R/2}(0)}  r(x)^{-1} |u(x)|  + \sup_{x \in \Sigma \setminus B_{R/2}(0)}  r(x) |E(x)| \\
& \qquad + \sup_{x,y \in \Sigma \setminus B_{R/2}(0)} \frac{1}{r(x)^{-1-\alpha}+r(y)^{-1-\alpha}} \frac{| E(x) - E(y)|}{|x-y|^{\alpha}} \\ \\
& \lesssim \Vert u \Vert_{C^{0}_{\textrm{hom};+1}(\Sigma \setminus B_{R/2}(0))} + \Vert E \Vert_{\cC\cS^{0,\alpha}_{-1}(\Sigma\setminus B_{R/2}(0))}.
\end{align*}

Arguing similarly for $Du$ and combining all of this with standard interior (elliptic) Schauder theory, we thus find
\begin{equation}\label{eq:parabolic-to-elliptic-first-est}
\Vert Du\Vert_{\cC\cS^{0,\alpha}_{-1}(\Sigma)} + \Vert D^{2}u \Vert_{ \cC\cS^{0,\alpha}_{-1}(\Sigma)} \lesssim \Vert u \Vert_{C^{0}_{\textrm{hom};+1}(\Sigma)} + \Vert E \Vert_{\cC\cS^{0,\alpha}_{-1}(\Sigma)}
 \end{equation}
Note that we can combine this inequality with an interpolation between $u$ in $C^{0}$ and $C^{1}$ to find
\[
\Vert u \Vert_{C^{0,\alpha}_{\textrm{hom};+1}(\Sigma)} \lesssim \Vert u \Vert_{C^{0}_{\textrm{hom};+1}(\Sigma)} + \Vert E \Vert_{\cC\cS^{0,\alpha}_{-1}(\Sigma)}.
\]
This allows us to bound $a u$ in $\cC\cS^{0,\alpha}_{-1}(\Sigma)$ in the sequel. 

We now argue that $u$ can be decomposed as $u(r,\omega) =\chi(r) c(\omega) r + f(r,\omega)$ making $u$ into an element of $\cC\cS^{2,\alpha}_{-1}(\Sigma)$. We have that
\[
w : = r\partial_{r} u - u= - 2 E + 2 \Delta u + 2 a u  - (\vec{x}\cdot \nabla_{\Sigma} u - r\partial_{r}u) .
\]
Combining \eqref{eq:parabolic-to-elliptic-first-est} with Lemma \ref{lemm:proj-radial-vect-Sigma}, we see that $w \in \cC\cS^{0,\alpha}_{-1}$ with
\[
\Vert w \Vert_{\cC\cS^{0,\alpha}_{-1}(\Sigma)} \lesssim \Vert u \Vert_{C^{0}_{\textrm{hom};+1}(\Sigma)} + \Vert E \Vert_{\cC\cS^{0,\alpha}_{-1}(\Sigma)}
\]
Now, we define
\begin{equation}\label{eq:def_c}
c(\omega) : = \frac{u(R,\omega)}{R} + \int_{R}^{\infty} \frac{w(s,\omega)}{s^{2}}ds,
\end{equation}
where $R$ is chosen large above (we emphasize that this expression is independent of the choice of $R$ and that the integral is finite, thanks to the fact that $w\in \cC\cS^{0,\alpha}_{-1}(\Sigma)$).  

We note that the functions
\[
 \omega \mapsto \frac{u(r,\omega)}{r}
\]
have uniformly bounded $C^{2,\alpha}(\Gamma)$ norm for $r$ sufficiently large. On one hand, they converge in $C^{0,\alpha}(\Gamma)$ to $c(\omega)$ by the previous analysis. On the other hand, by Arzel\'a--Ascoli, they converge in $C^{2,\beta}(\Gamma)$ (for any $\beta <\alpha$) to $c(\omega) \in C^{2,\alpha}(\Gamma)$, and we find that (by lower semicontinuity of the H\"older norm in this situation)
\[
\Vert c \Vert_{C^{2,\alpha}(\Gamma)} \leq \Vert D^{2}u \Vert_{\cC\cS^{0,\alpha}_{-1}(\Sigma)} \lesssim \Vert u \Vert_{C^{0}_{\textrm{hom};+1}(\Sigma)} + \Vert E \Vert_{\cC\cS^{0,\alpha}_{-1}(\Sigma)}
\]
where we again used \eqref{eq:parabolic-to-elliptic-first-est} in the second inequality. Now, defining
\[
f(r,\omega)  =\chi(r) c(\omega) r - u(r,\omega),
\]
we see that $f\in C^{2,\alpha}_{\textrm{loc}}(\Sigma)$. Note for $r$ sufficiently large we have from \eqref{eq:def_c} that
\[
f(r,\omega) = r \int_r^\infty \frac{w(s,\omega)}{s^{2}}ds\, ,
\]
which implies
\[
\Vert f \Vert_{\cC\cS^{0,\alpha}_{-1}(\Sigma)} \lesssim \Vert u \Vert_{C^{0}_{\textrm{hom};+1}(\Sigma)} + \Vert E \Vert_{\cC\cS^{0,\alpha}_{-1}(\Sigma)}\, .
\]
Moreover, using the estimates for $D^{2}u$ (and for $D^{2}(\chi(r)rc(\omega))$ which are easily derived from the $C^{2,\alpha}$ estimate for $c$), along with interpolation, we find that
\[
\Vert f \Vert_{C^{2,\alpha}_{\textrm{hom},-1}(\Sigma)} \lesssim \Vert u \Vert_{C^{0}_{\textrm{hom};+1}(\Sigma)} + \Vert E \Vert_{\cC\cS^{0,\alpha}_{-1}(\Sigma)}\, .
\]
Finally, it remains to estimate $\vec{x}\cdot \nabla_{\Sigma} f \in C^{0,\alpha}_{\textrm{hom},-1}(\Sigma)$. However, this follows from
\[
r \partial_{r} f = f - w
\]
and Lemma \ref{lemm:proj-radial-vect-Sigma}. This completes the proof. 
\end{proof}

\subsection{Weighted Sobolev spaces}\label{subsec:weighted:sob}
In this section, we combine the H\"older space theory developed above, with integral estimates and a Fredholm alternative to establish existence results for the $\cL_{\frac12}$ operator. The way to use these weighted Sobolev spaces to prove the Fredholm alternative (cf.\ Theorem \ref{thm:Lax-Milgram-consequence} below) was explained to us by J.\ Bernstein \cite{bernstein:private-Schauder}. 

We denote by $L^{2}_{W}$\index{$L^{2}_{W}(\Sigma)$} the space of measurable functions $f:\Sigma\to\RR$ with
\[
\Vert f \Vert_{W}^{2} := \int_{\Sigma} f^{2} \rho\,  d\cH^{n}<\infty. 
\]
We then define the Sobolev norm
\[
\Vert f \Vert_{W,k}^{2} : = \sum_{j=0}^{k} \Vert (\nabla_{\Sigma})^{j} f\Vert_{W}^{2}. 
\]
It is easy to see that the associated Sobolev space $H_{W}^{k}(\Sigma)$\index{$H_{W}^{k}(\Sigma)$}  is precisely the closure of $C^{\infty}_{0}(\Sigma)$ under this norm. 

We recall the following Sobolev inequlity due to Ecker \cite[p.\ 109]{Ecker:Sobolev} (see also \cite[Lemma B.1]{BernsteinWang:topological-AC-shrinkers}. 
\begin{prop}\label{prop:ecker-sobolev}
For  $f \in H_W^1(\Sigma)$, we have
\[
 \int_{\Sigma} f^{2} |x|^{2} \rho \, d\cH^{n} \leq 4 \int_{\Sigma}\left( n f^{2} + 4 |\nabla_{\Sigma} f|^{2} \right)\rho \, d\cH^{n}
\]
\end{prop}
\begin{proof}
Assume first $f \in C^{\infty}_{0}(\Sigma)$. Consider the vector field $V : = f^{2}\rho \vec{x}$ in the (Euclidean) first variation formula along $\Sigma$. We obtain
\[
\int_{\Sigma}\left(n f^{2} + 2 f  \vec{x}\cdot \nabla_{\Sigma} f - \frac 12 f^{2}  |x^{T}|^{2} \right)\rho \, d\cH^{n} = \int_{\Sigma} f^{2} H \bangle{x,\nu_{\Sigma}} \rho \, d\cH^{n}
\]
Using the shrinker equation, we thus find
\[
\int_{\Sigma}\left(n f^{2} + 2 f  \vec{x}\cdot \nabla_{\Sigma} f  \right)\rho \, d\cH^{n} = \frac 12 \int_{\Sigma} f^{2} |x|^{2} \rho \, d\cH^{n}
\]
Thus, we find that 
\[
\frac 12 \int_{\Sigma} f^{2} |x|^{2} \rho \, d\cH^{n} \leq \int_{\Sigma}\left(n f^{2} + 2 f  \vec{x}\cdot \nabla_{\Sigma} f \right)\rho \, d\cH^{n} \leq \int_{\Sigma}\left( n f^{2} + 4 |\nabla_{\Sigma} f|^{2} + \frac 1 4 |x|^{2} f^{2} \right)\rho \, d\cH^{n}
\]
Now let $f \in H_W^1(\Sigma)$ and choose $f_i\in C^\infty_0(\Sigma)$ such that $f_i \rightarrow f$ in $H_W^1(\Sigma)$. The above estimate yields for any $R>0$
$$ \int_{\Sigma\cap B_R(0)} f_i^{2} |x|^{2} \rho \, d\cH^{n} \leq 4 \int_{\Sigma}\left( n f_i^{2} + 4 |\nabla_{\Sigma} f_i|^{2} \right)\rho \, d\cH^{n}\, . $$
Letting $i \rightarrow \infty$ and then $R \rightarrow \infty$ yields the statement.
\end{proof}

\begin{coro}\label{coro:cL-sobolev-mapping}
The map $\cL_{\gamma} : H^{2}_{W}(\Sigma) \to L^{2}_{W}(\Sigma)$ is bounded. 
\end{coro}
\begin{proof}
Apply Ecker's Sobolev inequality to the gradient of $f$ to bound $\vec{x}\cdot\nabla_{\Sigma} f \in L^{2}_{W}(\Sigma)$. 
\end{proof}

\begin{lemm}[cf.\ {\cite[Proposition 3.4]{BernsteinWang:degreeExpanders}}]\label{lemm:H1-estimate-cL0}
For $f \in H^{2}_{W}(\Sigma)$, 
\[
\Vert f \Vert_{W,1}^{2} \leq \Vert \cL_{0}f \Vert_{W}\Vert f \Vert_{W}
\]
\end{lemm}
\begin{proof}
It suffices to prove this for $f\in C^{\infty}_{0}(\Sigma)$. Note that $\cL_{0}$ is self adjoint with respect to the Gaussian area. Thus,
\[
0 = \int_{\Sigma} \cL_{0}(f^{2}) \rho d\cH^{n} = 2 \int_{\Sigma} ( |\nabla f|^{2} + f\cL_{0}f ) \rho \, d\cH^{n}.
\]
This proves the claim.
\end{proof}

\begin{lemm}[{\cite[Proposition B.2]{BernsteinWang:topological-AC-shrinkers}}]\label{lem:compact_emb}
The inclusion $H^{1}_{W}\subset L^{2}_{W}$ is compact. 
\end{lemm}
\begin{proof}
For $f_{j} \in H_{W}^{1}$ with $\Vert f_{j}\Vert_{H_{W}^{1}} \leq C$, the classical Rellich compactness theorem applied to an exhaustion of $\Sigma$ shows that (after passing to a subsequence) there is $f \in H^{1}_{W}$ so that $f_{j}\to f$ in $L^{2}_{\textrm{loc}}$. That $f_{j}\to f$ follows easily from Ecker's Sobolev inequality, which implies that 
\[
\int_{\Sigma\setminus B_{\lambda}(0)} (f_{j}-f)^{2} \rho \, d\cH^{n} \lesssim \frac{C}{\lambda^{2}}. 
\]
This concludes the proof. 
\end{proof}

\begin{lemm}[cf.\ {\cite[Proposition 3.4]{BernsteinWang:degreeExpanders}}]\label{lem:H2_est}
For $f \in H^{2}_{W}(\Sigma)$, we have
\[
\Vert f\Vert_{W,2}^{2} \leq C (\Vert \cL_{0} f \Vert^{2}_{W} + \Vert f \Vert_{W}^{2})
\]
\end{lemm}
\begin{proof}
It suffices to prove this for $f\in C^{\infty}_{0}(\Sigma)$. Using the Bochner identity and the Gauss equations, we find (using $|A_{\Sigma}| = O(1)$)
\begin{align*}
 & \frac 12 \cL_{0} |\nabla_{\Sigma} f|^{2} \\
 & = |\nabla^{2} f|^{2} + \bangle{\nabla_{\Sigma}\Delta_{\Sigma}f ,\nabla_{\Sigma} f} + \Ric_{\Sigma}(\nabla_{\Sigma}f,\nabla_{\Sigma}f) - \frac{1}{4} \bangle{ x, \nabla |\nabla f|^2}\\
& = |\nabla^{2} f|^{2} + \bangle{\nabla_{\Sigma}\Delta_{\Sigma}f ,\nabla_{\Sigma} f}  - \frac{1}{4} \bangle{ x, \nabla |\nabla f|^2} \\
&\qquad+ H_{\Sigma} \cdot A_{\Sigma}(\nabla_{\Sigma}f,\nabla_{\Sigma}f)  - (A_{\Sigma})^{2}(\nabla_{\Sigma}f,\nabla_{\Sigma}f)  \\
& = |\nabla^{2} f|^{2} + \bangle{\nabla_{\Sigma}\cL_{0}f ,\nabla_{\Sigma} f} + \frac 12 \bangle{\nabla_{\Sigma}(\vec{x}\cdot\nabla_{\Sigma}f),\nabla_{\Sigma} f} - \frac{1}{4} \bangle{ x, \nabla |\nabla f|^2}\\
& \qquad + H_{\Sigma} \cdot A_{\Sigma}(\nabla_{\Sigma}f,\nabla_{\Sigma}f) - (A_{\Sigma})^{2}(\nabla_{\Sigma}f,\nabla_{\Sigma}f)  \\
& = |\nabla^{2} f|^{2} + \bangle{\nabla_{\Sigma}\cL_{0}f ,\nabla_{\Sigma} f} + \frac 12 |\nabla_{\Sigma}f|^{2} + H_{\Sigma} \cdot A_{\Sigma}(\nabla_{\Sigma}f,\nabla_{\Sigma}f)  - (A_{\Sigma})^{2}(\nabla_{\Sigma}f,\nabla_{\Sigma}f)\\
& = |\nabla^{2} f|^{2} + \bangle{\nabla_{\Sigma}\cL_{0}f ,\nabla_{\Sigma} f} + O(|\nabla_{\Sigma}f|^{2}). 
\end{align*}
Integrating this and using that $\cL_{0}$ is self adjoint with respect to the Gaussian area, the conclusion follows (after integrating by parts the second term in the right hand side, and using Lemma \ref{lemm:H1-estimate-cL0} to control the $H^{1}_{W}$ norm of $f$). 
\end{proof}

This suffices to establish an existence theory for the $L$ operator (cf.\ \cite[Proposition 3.4]{BernsteinWang:degreeExpanders}) where
\[
L := \cL_{\frac 12} + |A_{\Sigma}|^{2} = \Delta_{\Sigma} - \frac 12 (\vec{x}\cdot\nabla_{\Sigma} - 1) + |A_{\Sigma}|^{2}. 
\]
Define
\[
B_{\gamma}(u,v) := \int_{\Sigma} \left( \bangle{\nabla_{\Sigma} u,\nabla_{\Sigma} v} + \left( \gamma -  |A_{\Sigma}|^{2}  - \frac 12  \right) u v \right) \rho\, d\cH^{n},
\]
the bilinear form naturally associated to $L+\gamma$. For $\gamma$ sufficiently large so that $\gamma \geq \max_{\Sigma} |A_{\Sigma}|^{2} + \frac 32$, we see that
\[
\Vert u \Vert_{W,1}^{2} \leq B_{\gamma}(u,v),
\]
so $B_{\gamma}$ is coercive on $H^{1}_{W}(\Sigma)$. It is clearly bounded, so applying the Lax--Milgram Theorem, and applying the standard Fredholm alternative to this setting (combining Lemma \ref{lem:H2_est} with Lemma \ref{lem:compact_emb}), we have the following result:
\begin{theo}\label{thm:Lax-Milgram-consequence}
The space $\ker L \subset H^{1}_{W}$ of weak solutions to $Lu=0$ is finite dimensional. For $f\in L^{2}_{W}(\Sigma)$, $Lu=f$ has a weak solution in $H^{1}_{W}(\Sigma)$ if and only if $f$ is $L^{2}_{W}$-orthogonal to $\ker L$. Moreover, if $u$ is orthogonal to $\ker L$ and satisfies $Lu=f$, then we have the estimate $\Vert u \Vert_{H^{2}_{W}(\Sigma)} \leq C \Vert f \Vert_{L^{2}_{W}(\Sigma)}$. 
\end{theo}

To complete this section, we now show that for $f \in \cC\cS^{\alpha}_{-1}(\Sigma)$ perpendicular to $\ker L$, we can solve $Lu=f$. It remains to check that a solution of $Lu=f$ with $f \in \cC\cS^{0,\alpha}_{-1}(\Sigma)$ satisfies $u \in C^{0}_{\textrm{hom};+1}(\Sigma)$ a priori.
\begin{lemm}
For $f \in L^{2}_{W}(\Sigma) \cap C^{0}(\Sigma)$, if $Lu=f$ for $u \in H^{1}_{W}(\Sigma)$, then $u \in C^{0}_{\textnormal{hom};+1} (\Sigma)$ and for $R$ sufficiently large,
\[
\Vert u\Vert_{C^{0}_{\textnormal{hom};+1}(\Sigma)} \lesssim \Vert f \Vert_{C^{0}(\Sigma)} + \Vert u \Vert_{C^{0}(\Sigma\cap B_{R}(0))}.
\]
\end{lemm}
\begin{proof}
For $\varphi:\RR^{n+1}\to\RR$, we compute
\begin{align*}
L\varphi & =  \Delta_{\Sigma} \varphi - \frac 12 (\vec{x}\cdot \nabla_{\Sigma}\varphi - \varphi) + |A_{\Sigma}|^{2}\varphi\\
& = \Delta_{\RR^{n+1}} \varphi - D^{2}\varphi(\nu_{\Sigma},\nu_{\Sigma})  - H_{\Sigma} \bangle{\nu_{\Sigma}, \nabla_{\RR^{n+1}} \varphi} - \frac 12 (\vec{x}\cdot \nabla_{\Sigma}\varphi - \varphi) + |A_{\Sigma}|^{2}\varphi\\
& = \Delta_{\RR^{n+1}} \varphi - D^{2}\varphi(\nu_{\Sigma},\nu_{\Sigma}) - \frac 12 \bangle{\vec{x} , \nabla_{\RR^{n+1}} \varphi } + \frac12 \varphi + |A_{\Sigma}|^{2}\varphi.
\end{align*}
We consider $\varphi (x) = \alpha |x|  -\beta$. Then,
\begin{align*}
L \varphi & = \alpha \left( \frac{n-1}{|x|} - \frac{\bangle{\vec{x},\nu_{\Sigma}}^{2}}{|x|^{3}}  \right) - \frac 12\beta + O(|x|^{-2}) (\alpha |x|-\beta) \\
& \leq  - \frac 12 (1+O(|x|^{-2}))\beta + O(|x|^{-1})\alpha .
\end{align*}
Thus, $v=u-\varphi$ satisfies 
\[
L v \geq  f + \frac 12 (1+O(|x|^{-2}))\beta -  O(|x|^{-1})\alpha
\]
We fix $R\geq \underline R(\Sigma)$ and set
\begin{align*}
\alpha & = 2 \sup_{\Sigma} |f| + 2 R^{-1} \sup_{\Sigma \cap \partial B_{R}(0)} |u|\\
\beta & = 4 \sup_{\Sigma} |f| + R^{-1} \sup_{\Sigma \cap \partial B_{R}(0)} |u|. 
\end{align*}
This yields
\[
L v \geq (1-O(R^{-1}))\sup_{\Sigma} |f| + \frac 12  R^{-1} (1+O(R^{-1}))  \sup_{\Sigma \cap \partial B_{R}(0)} |u| > 0,
\]
for $R$ sufficiently large. Moreover, we find that 
\[
\sup_{\Sigma \cap \partial B_{R}} v \leq  -(1-R^{-1}) \sup_{\Sigma \cap \partial B_{R}} |u| - 2 (R- 2) \sup_{\Sigma} |f| <0
\]
 as long as $R$ is sufficiently large. Thus, we have arranged that $v\leq 0$ in a neighborhood of $\Sigma\cap\partial B_{R}(0)$. We now argue that $v\leq 0$ on $\Sigma\setminus B_{R}(0)$. Because $v^{+} \in H^{1}_{W}$, we find that
 \[
- \int_{\Sigma\setminus B_{R}(0)} |A_{\Sigma}|^{2} (v^{+})^{2} \rho \, d\cH^{n} \leq \int_{\Sigma\setminus B_{R}(0)} v^{+}\cL_{\frac 12} v \,\rho\, d\cH^{n} = \int_{\Sigma \setminus B_{R}(0)} \left( - |\nabla v^{+}|^{2} + \frac 12 (v^{+})^{2} \right) \rho \, d\cH^{n}
 \]
 Thus, using Ecker's Sobolev inequality, Proposition \ref{prop:ecker-sobolev}, we find that
 \[
 R^{2} \int_{\Sigma\setminus B_{R}(0)} (v^{+})^{2} \rho\,d\cH^{n} \leq (8 + 4n +O(R^{-2})) \int_{\Sigma\setminus B_{R}(0)} (v^{+})^{2} \rho\,d\cH^{n}.
 \]
 For $R$ sufficiently large, we thus see that $v^{+}\equiv 0$. Thus, $u \leq \varphi$ on $\Sigma\setminus B_{R}(0)$. Applying the same reasoning to $-u$ completes the proof. 
\end{proof}
Combining this estimate with Proposition \ref{prop:CS-schauder} we arrive at:

\begin{coro}
For $f \in \cC\cS^{0,\alpha}_{-1}(\Sigma)$, if $u \in H^{1}_{W}(\Sigma)$ satisfies $Lu = f$ weakly, then $u \in \cC\cS^{2,\alpha}_{-1}(\Sigma)$ and for $R>0$ fixed sufficiently large,
\[
\Vert u \Vert_{\cC\cS^{2,\alpha}_{-1}(\Sigma)}\lesssim \Vert u\Vert_{C^{0}(\Sigma\cap B_{R}(0))} + \Vert f\Vert_{\cC\cS^{0,\alpha}(\Sigma)}. 
\]
\end{coro}

Combined with Theorem \ref{thm:Lax-Milgram-consequence}, we thus see that the following standard solvability condition continues to hold in our setting:
\begin{coro}\label{coro:surjectivity-L-operator}
If $f \in \cC\cS^{0,\alpha}_{-1}(\Sigma)$, then we can find $u \in \cC\cS^{2,\alpha}_{-1}(\Sigma)$ solving $Lu=f$ if and only if $f$ is $L^{2}_{W}$-orthogonal to $\ker L \subset H^{1}_{W}(\Sigma)$. 
\end{coro}

\section{The \L ojasiewicz--Simon inequality for entire graphs}\label{sec:Loj-entire}
\noindent We now show that the weighted H\"older and Sobolev spaces considered in the previous section (along with the solvability criteria proven for $L$) provides a framework to prove the \L ojasiewicz--Simon inequality following the arguments in the compact case (cf.\ \cite{Simon:Loj,Simon:reg-sing-harm-maps,Schulze:compact,zemas}). By the Fredholm alternative, Theorem \ref{thm:Lax-Milgram-consequence}, $\ker L \subset H^{1}_{W}(\Sigma)$ is finite dimensional and we can define $\Pi : L^{2}_{W} (\Sigma) \to L^{2}_{W} (\Sigma) \cap \cC\cS^{2,\alpha}_{-1}(\Sigma)$, the projection on to $\ker L$. 

Recall (see Appendix \ref{app:EL}) that the Euler--Lagrange equation (with respect to the $L^{2}_{W}$-inner product) is
\begin{equation}\label{eq:M_def}
\cM(v) =  \Pi_{T^\perp\Sigma}\left(\vec{H}_{M} + \frac{x^{\perp}}{2} \right)\Big|_{x=y+v(y)\nu_{\Sigma}(y)}  J(y,v,\nabla_{\Sigma}v)\rho(y+v(y)\nu_{\Sigma})\rho(y)^{-1} 
\end{equation}
where $\Pi_{T^\perp\Sigma}$\index{$\Pi_{T^\perp\Sigma}$} is the projection on to the normal bundle to $\Sigma$ and \index{$J$}
\[
J(y,v,\nabla_{\Sigma}v) = \Jac(D \exp_{y}(v(y)\nu_{\Sigma}(y)))
\]
 is the area element. 

We now observe that $\cM$ is a well behaved map between the weighted H\"older spaces considered in the previous section. 

\begin{lemm}\label{lemm:M-is-well-defined-diff}
For $\beta$\index{$\beta$} sufficiently small depending on $\Sigma$, we have a continuous map
\[
\cM : \cC\cS^{2,\alpha}_{-1}(\Sigma) \cap \{ \Vert u \Vert_{\cC\cS^{2,\alpha}_{-1}(\Sigma)} < \beta\} \to \cC\cS^{0,\alpha}_{-1}(\Sigma).
\]
Moreover, $\cM$ is Fr\'echet differentiable with derivative at $0$ given by $L$. 
\end{lemm}
\begin{proof}
Fix $v\in \cC\cS^{2,\alpha}_{-1}(\Sigma) \cap \{ \Vert u \Vert_{\cC\cS^{2,\alpha}_{-1}(\Sigma)}<\beta\}$. Note that 
\begin{equation}\label{eq:exponential-factor-change}
J(y,v,\nabla_{\Sigma}v)\rho(y+v(y)\nu_{\Sigma})\rho(y)^{-1} = J(y,v,\nabla_{\Sigma}v) \exp\left({- \frac{2 v(y) \bangle{y,\nu_\Sigma} + (v(y))^2}{4}}\right) 
\end{equation}
and this is easily seen to be uniformly bounded in $C^{0,\alpha}(\Sigma\cap B_{1}(y))$  as $y \in\Sigma\to\infty$. Thus, it remains to check the first term. Observe that the mean curvature term is uniformly bounded in $C^{0,\alpha}(\Sigma\cap B_{1}(y))$ by $c/r$ as $y \in\Sigma\to\infty$. Recall that differentiating the shrinker equation yields (or see Lemmas \ref{lemm:second-fund-form} and \ref{lemm:proj-radial-vect-Sigma})
\[ A(x^T,\cdot) = O(r^{-2})\, .
\]
 Combining this with \eqref{eq:ng.1} and the shrinker equation for $\Sigma$ we get for the other term that
\[
\langle x, \nu \rangle= (1 + |(\text{Id} - vS)^{-1}(\nabla^M v)|^2)^{-\frac{1}{2}} \big(v - \bangle{y,\nabla_{\Sigma}v} \big) + O( |y|^{-1} ),
\]
in $C^{0,\alpha}(\Sigma\cap B_{1}(y))$ as $y\to\infty$. Observing that $v\mapsto v - \bangle{y,\nabla_{\Sigma}v}$ is a bounded map $\cC\cS^{2,\alpha}_{-1}(\Sigma)\to\cC\cS^{0,\alpha}_{-1}(\Sigma)$) we obtain the first assertion. The second follows similarly. 
\end{proof}

We define
\[
\cN : = \cM + \Pi
\]
which has the same mapping properties as $\cM$. Moreover, $\cN$ is Fr\'echet differentiable with derivative at $0$ given by $L+\Pi$ (which is bijective as a linear map $\cC\cS^{2,\alpha}_{-1}(\Sigma)\to\cC\cS^{0,\alpha}_{-1}(\Sigma)$).  Thus, the implicit function theorem allows us to find open neighborhoods of $0$,
\begin{align*}
W_{1} & \subset \cC\cS^{2,\alpha}_{-1}(\Sigma) \cap \{ \Vert u \Vert_{\cC\cS^{2,\alpha}_{-1}(\Sigma)} < \beta\}\\
W_{2} & \subset \cC\cS^{0,\alpha}_{-1}(\Sigma)
\end{align*}
so that $\cN : W_{1}\to W_{2}$ is bijective with inverse $\Psi: W_{2}\to W_{1}$. Moreover (cf.\ \cite[\S 3.13]{Simon:green-book} and \cite[p.\ 168]{Schulze:compact}), $\cN$ and $\Psi$ are holomorphic, after tensoring with $\CC$ (and possibly shrinking $W_{1},W_{2}$). 

We now prove that $\cM$ is continuous as a map $H^{2}_{W}\cap W_{1}\to L^{2}_{W}$ and that $\Psi$ is continuous as a map $L^{2}_{W} \cap W_{2} \to H^{2}_{W}$. 
\begin{lemm}\label{lemm:cont-psi}
Shrinking $W_{1},W_{2}$ if necessary, there is $C>0$ so that
\[
\Vert \cM(u_{1})-\cM(u_{2})\Vert_{L^{2}_{W}(\Sigma)} \leq C \Vert u_{1}-u_{2}\Vert_{H^{2}_{W}(\Sigma)}
\]
for $u_{1},u_{2} \in W_{1}$ and moreover
\[
\Vert \Psi(f_{1}) - \Psi(f_{2})\Vert_{H_{W}^{2}(\Sigma)} \leq C \Vert f_{1}-f_{2} \Vert _{L^{2}_{W}(\Sigma)}
\]
for $f_{1},f_{2}\in W_{2}$. 
\end{lemm}
\begin{proof}
We claim that
\begin{equation}\label{eq:diff-cMs}
\cM(u_{1}) - \cM(u_{2}) = L(u_{1} - u_{2}) + A \cdot \nabla^{2}(u_{1}-u_{2}) + B\cdot\nabla (u_{1}-u_{2}) + C(u_{1}-u_{2})
\end{equation}
where
\begin{equation}\label{eq:diff-cMs-est}
\sup_{\Sigma} (|A| + |B| + |C|) \lesssim \Vert u_{1}\Vert_{\cC\cS^{2,\alpha}_{-1}(\Sigma)}+ \Vert u_{2}\Vert_{\cC\cS^{2,\alpha}_{-1}(\Sigma)}.
\end{equation}
This follows from using \eqref{eq:M_def}, \eqref{eq:ng.2}, \eqref{eq:ng.3} and \eqref{eq:ng.4} together with the shrinker equation along $\Sigma$ to write 
$$ \cM(u) = L u + Q(p,u,\nabla u, \nabla^2u) $$
and interpolating $Q(p,u,\nabla u, \nabla^2u)$ in the standard way between $u_1$ and $u_2$. Combined with Corollary \ref{coro:cL-sobolev-mapping}, this proves the first assertion. The second claim now follows from standard arguments (cf.\ \cite[\S 3.12]{Simon:green-book}) given \eqref{eq:diff-cMs}, \eqref{eq:diff-cMs-est}, and Theorem \ref{thm:Lax-Milgram-consequence}. \end{proof}

At this point, we can follow the arguments in \cite[\S 3.11-3.13]{Simon:green-book} essentially verbatim (except we use Corollary \ref{coro:surjectivity-L-operator}, Theorem \ref{thm:Lax-Milgram-consequence}, Lemma \ref{lemm:M-is-well-defined-diff}, and Lemma \ref{lemm:cont-psi} in place of their standard counterparts in the compact case) to prove 
\begin{theo}[\L ojasiewicz--Simon inequality for entire graphs]\label{theo:loj-entire}
There is $\beta_{0}>0$\index{$\beta_{0}$} sufficiently small, $\theta \in (0,\frac 12)$,\index{$\theta$} and $C>0$, all depending on $\Sigma$, so that if $M$ is the graph over $\Sigma$ of a function in $u \in \cC\cS^{2,\alpha}_{-1}(\Sigma)$ with $\Vert u \Vert_{\cC\cS^{2,\alpha}_{-1}(\Sigma)} < \beta_{0}$, then 
\[
|F(M) - F(\Sigma)|^{1-\theta} \leq C \Vert \cM(u)\Vert_{L^{2}_{W}(\Sigma)} \leq  C\left(\int_{M} |\phi|^{2} \rho\, d\cH^{n} \right)^{\frac 12}. 
\]
\end{theo}
We note that the second inequality here follows a similar reasoning to \eqref{eq:exponential-factor-change} (so as to control the change in $\rho$ when evaluated along $M$ and as opposed to $\Sigma$).

\section{Defining the relevant scales}\label{sec:scales}

In order to apply the inequality obtained in Theorem \ref{theo:loj-entire}, we must understand the various geometric scales involved.
\subsection{Pseudolocality and the scale of the core of the shrinker} 
These definitions are relevant to the pseudolocality based improvement argument in Lemma \ref{lemm:rough-improves}. 

\begin{prop}[{Pseudolocality \cite[Theorem 1.5]{IlmanenSchulzeNeves}}]\label{prop:pseudolocality-result}
Given $\delta>0$, there exists $\gamma>0$ and a constant $\rho = \rho(n,\delta)\in (0,\infty)$ such that if a mean curvature flow $\{M_{t}\}_{t \in[-1,0]}$ satisfies that $M_{-1}\cap B_{\rho}(0)$ is a Lipschitz graph over the plane $\{x_{n+1} = 0\}$ with Lipschitz constant less than $\gamma$ and $0 \in M_{-1}$, then $M_t \cap B_\rho(0)$ intersects $B_\delta(0)$ and $M_t \cap B_\delta(0)$ remains a Lipschitz graph over $\{x_{n+1} = 0\}$ with Lipschitz constant less than $\delta$ for all $t \in [-1,0]$. 
\end{prop}

\begin{defi}[Fixing the Pseudolocality constants]
We will fix $\delta = 10^{-2}$ in the preceding Pseudolocality result. We denote the corresponding $\gamma$ by $\gamma_*$\index{$\gamma_{*}$} and $\rho=\rho_{*}$\index{$\rho_{*}$}. For consistency, we also write $\delta_{*}=\delta$\index{$\delta_{*}$}. We will always assume that $\rho_{*}\geq 1$. 
\end{defi}

\begin{defi}[Scale of the core of the conical shrinker]
For an asymptotically conical self-shrinker $\Sigma^{n}\subset \RR^{n+1}$, we choose $\underline R(\Sigma)$\index{$\underline R(\Sigma)$} so that for $x \in \Sigma\setminus B_{\underline R(\Sigma)}(0)$, we have that $\Sigma\cap B_{2\rho_{*}}(x)$ is a Lipschitz graph over $T_{x}\Sigma$ with Lipchitz constant less than $\gamma_{*}/2$. Furthermore, we require that the map from the end of $\cC$ described in Lemma \ref{lemm:improved-conical-est-shrinker} is defined outside of $B_{\underline R(\Sigma)-1}$. 
\end{defi}

It is clear that for an asymptotically conical shrinker, we may take $\underline R(\Sigma) < \infty$. 
\subsection{Scales of hypersurfaces near the shrinker} 
The definitions here are relevant to the radius at which one can apply a cut-off version of Theorem \ref{theo:loj-entire}. 

\begin{defi}[Shrinker scale]\label{defi:shrinker-scale}
For $M^{n}\subset \RR^{n+1}$ we define the \emph{shrinker scale}\index{shrinker scale} $\bR(M)$\index{$\bR(M)$} by
\begin{equation}\label{eq:shrinker-scale}
e^{-\frac{\bR(M)^{2}}{4}} : = |\nabla_{M}F|^{2} = \int_{M} |\phi|^{2} \rho \, d\cH^{n}\, .
\end{equation}
\end{defi}

\begin{defi}[Rough conical scale]\label{defi:rough-conical-scale}
For $M^{n}\subset \RR^{n+1}$, $\ell\in \NN$\index{$\ell$}, and $C_{\ell}>0$\index{$C_{\ell}$} we define the \emph{rough conical scale}\index{rough conical scale} $\tilde {\br}_{\ell}(M)$\index{$\tilde {\br}_{\ell}(M)$} to be the largest radius so that $M^{n}\cap B_{\tilde \br_{\ell}(M)}(0)$ is smooth and
\[
|\nabla^{(k)}A_{M}| \leq C_{\ell}(1+r)^{-1-k}
\]
for $k \in \{0,\dots,\ell+1\}$. 
\end{defi}

\begin{defi}[Conical scale]\label{defi:conical-scale}
Fix an asymptotically conical self-shrinker $\Sigma^{n}\subset \RR^{n+1}$ and choose $\beta_{0} =\beta_{0}(\Sigma)>0$ as in Theorem \ref{theo:loj-entire}. For a hypersurface $M^{n} \subset \RR^{n+1}$ we define the \emph{conical scale}\index{conical scale} $\br_{\ell}(M)$\index{$\br_{\ell}(M)$} to be largest radius in $[\underline R(\Sigma), \tilde \br_{\ell}(M)]$ so that there is $u : \Sigma \to\RR$ with
\[
\Graph u|_{\Sigma\cap B_{\br_{\ell}(M)} } \subset M \qquad \text{and} \qquad M \cap B_{\br_{\ell}(M)-1}\subset \Graph u,
\]
where $u \in \cC\cS^{2,\alpha}_{-1}(\Sigma)$ with $\Vert u\Vert_{\cC\cS^{2,\alpha}_{-1}(\Sigma)} < \beta_{0}$. 
\end{defi}

\begin{defi}[Core graphical hypothesis]\label{def:core-graph-hypoth}
 We say that $M$ satisfies the \emph{core graphical hypothesis}\index{core graphical hypothesis}, denoted by $(*_{b,\underline r})$\index{$(*_{b,\underline r})$}, if $\tilde \br_{\ell}(M)\geq  \underline r$ and there is $u : \Sigma \cap B_{\underline r} (0) \to \RR$ so that 
\[
\Graph u \subset M \qquad \text{and} \qquad M \cap B_{{ \underline r}-1}\subset \Graph u
\]
and $\Vert u \Vert_{C^{\ell+1}(B_{\underline r}(0))}\leq b$. 

We will always assume that \index{\underline r}$\underline r > \sqrt{2n}$ (so that $\partial B_{\underline r}(0)$ expands under the rescaled mean curvature flow). 
\end{defi}
We fix $b>0$\index{$b$} to be very small (e.g. $b\ll\beta_{0}$) in Proposition \ref{prop:approx-up-to-rough}.

\section{Localizing the \L ojasiewicz--Simon inequality}\label{sec:loc-LS}

We now localize Theorem \ref{theo:loj-entire} to hypersurfaces that are not entire graphs over $\Sigma$. {For the definition of $\lambda(M)$ see Definition \ref{defi:gaussian-area}.}

\begin{theo}[The local \L ojasiewicz--Simon inequality]\label{theo:cutoff-loj}
For $M^{n}\subset \RR^{n+1}$ with $\lambda(M) \leq \lambda_{0}$, $\gamma \in (1,2)$, and $R \in [1,\br_{\ell}(M)-1]$, we have that
\[
|F(M) - F(\Sigma)| \leq C \left( \left( \int_{M\cap B_{R}(0)} |\phi|^{2} \rho\, d\cH^{n} \right)^{\frac{1}{2(1-\theta)}} + R^{\frac{n-4}{2(1-\theta)}} e^{-\frac{R^{2}}{8(1-\theta)}} + e^{-\frac{R^{2}}{4\gamma}} \right)
\]
for $C=C(\Sigma,\lambda_{0},\alpha,\gamma)$. Here $\theta \in (0,\frac 12)$ depends on $\Sigma$ and the H\"older coefficient $\alpha$; $\theta$ is fixed in Theorem \ref{theo:loj-entire}. 
\end{theo}
\begin{proof}
By definition of $\br_{\ell}(\Sigma)$ (Definition \ref{defi:conical-scale}), there is $u:\Sigma\to\RR$ with 
\[
\Graph u |_{\Sigma\cap B_{R}(0)} \subset M \qquad \text{and} \qquad M \cap B_{R}(0) \subset \Graph u 
\]
with $\Vert u\Vert_{\cC\cS^{2,\alpha}_{-1}(\Sigma)} < \beta_{0}$. We may thus apply Theorem \ref{theo:loj-entire} to $\Graph u$ to obtain (allowing the constant $C$ to change from line to line as usual)
\begin{align*}
|F(M) - F(\Sigma)| & = \left| \int_{M} \rho\, d\cH^{n} - F(\Sigma)\right|\\
& \leq \left| \int_{M\cap B_{R}(0)} \rho d\cH^{n} - F(\Sigma)\right| + C e^{-\frac{R^{2}}{4\gamma}}\\
& \leq \left| \int_{\Graph u} \rho d\cH^{n} - F(\Sigma)\right| + C e^{-\frac{R^{2}}{4\gamma}}\\
& \leq  C \left(  \int_{\Graph u} |\phi|^{2} \rho \, d\cH^{n} \right)^{\frac{1}{2(1-\theta)}} + C e^{-\frac{R^{2}}{4\gamma}}.
\end{align*}
It remains to argue that we can restrict the first integral to $\Sigma\cap B_{R}(0)$. It is easy to see that $r|\phi_{\Graph u}| \leq C\beta_{0}$ by definition of $\cC\cS^{2,\alpha}_{-1}(\Sigma)$. Using
\[
\int_{R}^{\infty} r^{n-3} e^{-\frac{r^{2}}{4}} dr \lesssim R^{n-4} e^{-\frac{R^{2}}{4}},
\]
we thus obtain
\begin{align*}
|F(M) - F(\Sigma)| & \leq  C \left( \int_{(\Graph u) \cap B_{R}(0)} |\phi|^{2} \rho \, d\cH^{n} \right)^{\frac{1}{2(1-\theta)}} +C R^{\frac{n-4}{2(1-\theta)}} e^{-\frac{R^{2}}{8(1-\theta)}} + C e^{-\frac{R^{2}}{4\gamma}}.
\end{align*}
This completes the proof. 
\end{proof}

\section{Approximate shrinkers up to the rough conical scale} \label{sec:improvement-argument}

For $\theta$ fixed in Theorem \ref{theo:loj-entire}, define \index{$\Theta$}
\[
\Theta = \left(\frac{1-\frac\theta 2}{1-\theta}\right)^{\frac 1 4} \in \left(1,\left(\frac 32\right)^{\frac 14} \, \right].
\]

\begin{defi}\label{defi:rough-conical-approx-shrinker}
For $R \geq \underline r$, we say that $M^{n}\subset \RR^{n+1}$ is a \emph{roughly conical approximate shrinker up to scale $R$}\index{roughly conical approximate shrinker} if:
\begin{enumerate}
\item we have $\Theta R\leq  \tilde \br_{\ell}(M)$,
\item $M$ satisfies the core graphical hypothesis $(*_{b, \underline r})$, and
\item $|\phi| + (1+|x|)|\nabla \phi| \leq s (1+|x|)^{-1}$ on $M\cap B_{\Theta R}(0)$. 
\end{enumerate}
\end{defi}
We will fix $s,b$ sufficiently small in the following proposition giving a lower bound on the conical scale. \index{$s$}

\begin{prop} \label{prop:approx-up-to-rough}
Taking $\ell$ sufficiently large, there are constants $b,s>0$ sufficiently small, depending on the shrinker $\Sigma$, the conical scale constant $\beta_{0}$, the rough conical scale constant $C_{\ell}$, and the entropy bound $\lambda_0$ with the following property.
If $M^{n}\subset \RR^{n+1}$ has $\lambda(M) \leq \lambda_0$ and is a roughly conical approximate shrinker up to scale $R$ in the sense of Definition \ref{defi:rough-conical-approx-shrinker}, then there is a function $u: \Sigma \to\RR$ with
\[
\Graph u|_{\Sigma\cap B_{R}(0)} \subset M \qquad \text{and} \qquad M\cap B_{R-1}(0)\subset \Graph u
\]
and $\Vert u \Vert_{\cC\cS^{2,\alpha}_{-1}(\Sigma)} \leq \beta_{0}$. Equivalently, the conical scale satisfies $\br_{\ell}(M) \geq R$.  
\end{prop}
Certain aspects of the following proof are inspired by the proof of \cite[Theorem 8.9]{KKM:shrinkers}. 
\begin{proof}
We claim that for $b,s$ sufficiently small, the conclusion eventually holds for any $R\geq  \underline r$. As such, we will take $b,s\to 0$ and will prove that for any given (sequence) of $R \geq \underline r$, the conclusion eventually holds for $R$. We may assume that $R\to\infty$ (the subsequent argument is easily modified to the case where $R$ is bounded). It is clear that $M$ converges to $\Sigma$ in $C^{\ell}$ in $B_{\underline r-1}$ with multiplicity one. Moreover, $M$ converges in $C^{\ell}_{\textrm{loc}}(\RR^{n+1})$ to\footnote{By Lemma \ref{lemm:poly-area-growth}, $M'$ is a properly embedded hypersurface.} $M'$, which satisfies $\phi \equiv 0$, and is thus a properly\footnote{Properness of $M'$ follows from the entropy bounds, which imply local area bounds by Lemma \ref{lemm:poly-area-growth}.} embedded shrinker. Unique continuation implies that\footnote{Note that there cannot be more than one component of $M'$. One way to see this is that it would have to lie outside of $B_{\underline r}(0)$ and we chose $\underline r > \sqrt{2n}$; this would contradict the maximum principle. Alternatively this follows from the Frankel property for shrinkers (i.e., two properly embedded shrinkers must intersect); see \cite[Corollary C.4]{genericMCF} for Ilmanen's proof of this fact.} $M'=\Sigma$. Finally, it is clear that $M$ converges to $\Sigma$ in $C^{\ell}$ with multiplicity one everywhere by connectedness of $\Sigma$ and the multiplicity one convergence on $B_{\underline r-1}$. 

Hence, if we let $\hat R \in [\underline r,R]$ denote the largest radius (depending on $b,s$) so that the conclusion holds with $\hat R$ (in the place of $R$), it is clear that $\hat R\to\infty$. We will prove that the proposition holds up to $\tilde R : = \frac 12 (1+\Theta) \hat R$ (note that this is a fixed factor less than $\Theta \hat R$). This will then imply the claim by a straightforward contradiction argument. 

First of all, we can assume that $R/\hat{R} \rightarrow \lambda \in [1, \infty]$. Observe that $(\hat R)^{-1}M$ converges in $C^{\ell}_{loc}(B_{\Theta\lambda}(0)\setminus\{0\})$ to a cone $\hat \cC$ which is a $C^{\ell}$ graph over the original cone $\cC$. Moreover, because we have assumed that the proposition holds up to $\hat R$, we see that the cones are close in the sense that $d_{H}(\cC,\hat \cC) = O(\beta_{0})$.\footnote{We have written $d_{H}$ for the Hausdorff distance in $\SS^{n-1}$ between the two links of the cones.} Thus, we can find a $C^{\ell}$ function $u: \Sigma \cap B_{\tilde R}(0)\to\RR$ with 
\[
\Graph u \subset M \qquad \text{and} \qquad M \cap B_{\tilde R-1}\subset \Graph u.
\]
Moreover, $r^{-1}|u| \leq O(\beta_{0})$ on $\Sigma\cap (B_{\tilde R}(0)\setminus B_{\underline r}(0))$ by the above observation that the blow-down cones are $O(\beta_{0})$-close. Furthermore, the second fundamental form estimates coming from the rough conical scale estimate $\tilde \br_{\ell}(M) \geq \Theta \tilde R$ yield 
\begin{equation}\label{eq:higher-der-rough-conical-u}
|D^{2+k}u| = O(r^{-1-k})
\end{equation}
on $\Sigma\cap (B_{\Theta \tilde R}(0)\setminus B_{\underline r}(0))$, for $k \in \{0,\dots,\ell-1\}$. Finally,  because $M$ converges in $C^{\ell}_{\text{loc}}(\RR^{n+1})$ to $\Sigma$ (as $b,s\to0$), for $\delta\in (0,\underline r^{-1})$ fixed sufficiently small depending only on $\beta_{0}$ (this will be made explicit in the last line of the proof), we can assume that $\Vert u \Vert_{C^{3}(\Sigma \cap B_{2\delta^{-1}}(0))} \leq \delta^{3}$. 

We now relate the smallness condition on $\phi$ to decay properties of $u$. These computations are similar to those considered in Section \ref{subsec:improved-conical-est} for an exact shrinker (except we are now parametrizing $M$ over the shrinker $\Sigma$, rather than parametrizing the end of $\Sigma$ over the cone $\cC$; this complicates certain aspects of the subsequent computation).

We write $F(p) = p +u(p)\nu_{\Sigma}(p)$ for the function parametrizing (part of) $M$ over $\Sigma \cap B_{\tilde R}(0)$. The computations below will hold for $p\in\Sigma$ with $|p| \in [\underline R(\Sigma),\tilde R]$, with error terms uniform with respect to $b,s\to0$. Recall that we have fixed coordinates $(r,\omega)$ on $\Sigma \setminus B_{\underline R(\Sigma)}$ in Section \ref{sec:prelim}. In particular, the vector fields $\partial_r$ and $\partial_{\omega_j}$ are tangent to $\Sigma$.\footnote{In particular, we reiterate that the vector field $\partial_r$ is \emph{not} the Euclidean radial vector field!} 

We write
\[
\nu_{M}(F(p)) = A\partial_{r} + \sum_{j=1}^{n-1}B_{j}r^{-1}\partial_{\omega_{j}} + C\nu_{\Sigma}(p),
\]
where
\begin{equation}\label{eq:approxshrinker-ABC}
A^{2} + \sum_{j=1}^{n-1} B_{j}^{2} + C^{2} = 1+O(r^{-2}) 
\end{equation}
by Corollary \ref{coro:metric-conical-shrinker} (we emphasize that $(r,\omega)$ are the coordinates induced on the end of $\Sigma$ by the parametrization over $\cC$ constructed in Lemma \ref{lemm:improved-conical-est-shrinker}). 

 Moreover, we find for $p\in\Sigma$ with $|p|$ sufficiently large (assuming that $\omega_{j}$ are normal coordinates at $\omega$ for $p=(r,\omega)$) we find
\begin{align*}
0 & = A (1+O(r^{-2}) - u(p) A_{\Sigma}|_{p}(\partial_{r},\partial_{r})) \\
& \qquad + \sum_{j=1}^{n-1} B_{j} (O(r^{-2}) - u(p)A_{\Sigma}|_{p}(\partial_{r},r^{-1}\partial_{\omega_{j}})) + C(\partial_{r}u(p))\\
0 & = A(O(r^{-2}) -u(p) A_{\Sigma}|_{p}(\partial_{r},r^{-1}\partial_{\omega_{i}})) \\
& \qquad + \sum_{j=1}^{n-1} B_{j} (\delta_{ij} + O(r^{-2}) - u(p) A_{\Sigma}|_{p}(r^{-1}\partial_{\omega_{i}},r^{-1}\partial_{\omega_{j}})) + C(r^{-1}\partial_{\omega_{i}}u(p)).
\end{align*}
Now, using Lemma \ref{lemm:second-fund-form}, we find that 
\begin{equation}
\label{eq:eqn-normal-approx-shrinker-A} 0 = A (1+O(r^{-2})) + \sum_{j=1}^{n-1} B_{j} (O(r^{-2})) + C(\partial_{r}u(p))
\end{equation}
and
\begin{equation}
\label{eq:eqn-normal-approx-shrinker-B} 0 = A(O(r^{-2}))  + \sum_{j=1}^{n-1} B_{j} (\delta_{ij} - u(p) A_{\cC}|_{p}(r^{-1}\partial_{\omega_{i}},r^{-1}\partial_{\omega_{j}}) + O(r^{-2 +\eta})  ) + C(r^{-1}\partial_{\omega_{i}}u(p)).
\end{equation}

Observe that \eqref{eq:eqn-normal-approx-shrinker-A} yields (since $A,B,C=O(1)$)
\[
A + C (\partial_{r}u) = O(r^{-2}). 
\]
Moreover, as long as $\beta_{0}$ is sufficiently small so that that $r^{-1}|u(p)|\sup_{\Gamma}|A_{\Gamma}| \leq \frac 12$, we see that $C^{-1} = O(1)$, i.e., $C$ is not tending to zero.\footnote{Indeed, if $C\to 0$, then this condition on $\beta_{0}$ combined with \eqref{eq:eqn-normal-approx-shrinker-B} yields $B_{i}\to 0$ as well; returning to \eqref{eq:eqn-normal-approx-shrinker-A} yields $A\to 0$, which contradicts \eqref{eq:approxshrinker-ABC}.} 

We now compute
\begin{equation}\label{eq:support-function-approximate-shrinker-coarse}
\begin{split}
\bangle{F(p),\nu_{M}(F(p))} & = \bangle{p+u(p)\nu_{\Sigma}(p), A \partial_{r} + \sum_{j=1}^{n-1} B_{j} r^{-1}\partial_{\omega_{j}} + C \nu_{\Sigma}(p) }\\ 
& = A \bangle{p,\partial_{r}} + \sum_{j=1}^{n-1} B_{j} \bangle{p,r^{-1}\partial_{\omega_{j}}} + C u(p) + C\bangle{p,\nu_{\Sigma}(p)}\\
& =  A (r +O(r^{-1})) + C u(p) + 2C H_{\Sigma}(p) + O(r^{-1})\\ 
& =  C (u(p) - (r+O(r^{-1}))\partial_{r}u(p)) + O(r^{-1})\\ 
& =  C (u(p) - r \partial_{r}u(p)) + O(r^{-1}) \partial_{r}u(p) + O(r^{-1}).
\end{split}
\end{equation}

We begin by analyzing this expression (below, we will repeat the above derivation to yield more precise estimates). We have that 
\[
\bangle{F(p),\nu_{M}(F(p))} = 2 \phi(F(p)) + 2H_{M}(F(p)) = O(r^{-1}).
\]
Thus, \eqref{eq:support-function-approximate-shrinker-coarse} (and $C^{-1}=O(1)$) gives
\[
r \partial_{r}u(p) - u(p) = O(r^{-1})\partial_{r}u(p) + O(r^{-1}).
\]
Thus, 
\begin{align*}
\partial_{r} \frac{u(p)}{r} & = r^{-2}(r\partial_{r} u(p) - u(p)) \\
& = O(r^{-3})\partial_{r}u(p) + O(r^{-3})\\
& = O(r^{-2})\partial_{r} \frac{u(p)}{r} - O(r^{-4}) u(p) + O(r^{-3}). 
\end{align*}
Thus, using $r^{-1} u = O(1)$, we conclude that 
\[
\partial_{r} \frac{u(p)}{r} = O(r^{-3}). 
\]
We integrate this from $\delta^{-1}$ to $r \in (\delta^{-1},\tilde R]$ to find 
\[
\frac{u(r,\omega)}{r} = \delta u(\delta^{-1},\omega) + O(\delta^{2}) + O(r^{-2}) = O(\delta^{2}) + O(r^{-2}),
\]
using the fact that $\Vert u \Vert_{C^{3}(\Sigma \cap B_{\delta^{-1}}(0))} \leq \delta^{3}$. Thus,
\begin{equation}\label{eq:u-delta-smallness-radial-growth}
u = O(\delta^{2}) r + O(r^{-1}). 
\end{equation}
Note that we immediately get
\[
\partial_{r} u(p) = r^{-1}u(p) + O(r^{-2}) = O(\delta^{2}) + O(r^{-2}).
\]

We now interpolate \eqref{eq:u-delta-smallness-radial-growth} (on balls of radius $1$) with the higher derivative estimates from \eqref{eq:higher-der-rough-conical-u}, using Lemma \ref{lemm:interpolation}. This yields 
\begin{align*}
|D^{2}u| & \lesssim (O(\delta^{2}) r + O(r^{-1}))^{1-\frac 2 \ell} r^{(1-\ell)\frac{2}{\ell}}\\
& = O(\delta^{2-\frac 4 \ell}) r^{-1} + O(r^{\frac 4 \ell - 3}).
\end{align*}
Similarly, we can obtain an estimate for the full gradient
\begin{align*}
|Du| & \lesssim (O(\delta^{2}) r + O(r^{-1}))^{1-\frac 1 \ell} r^{(1-\ell)\frac{1}{\ell}}\\
& = O(\delta^{2-\frac 2\ell}) + O(r^{\frac2\ell - 2})
\end{align*}

Now we return to \eqref{eq:support-function-approximate-shrinker-coarse} and use this improved decay for the derivatives to derive a sharper equation. Firstly, we note that as long as $\beta_{0}$ is small, as above, using the gradient estimate for $u$, \eqref{eq:eqn-normal-approx-shrinker-B}, together with \eqref{eq:u-delta-smallness-radial-growth}, implies that
\[
B_{i} = O(\delta^{2-\frac 2\ell}),
\]
for $r\geq \delta^{-1}$.
Finally, using this, $A + C (\partial_{r}u) = O(r^{-2})$, and \eqref{eq:approxshrinker-ABC}, we find that 
\[
C = 1 + O(\delta^{2-\frac{2}{\ell}}),
\]
for $r\geq \delta^{-1}$.

Now, repeating the derivation used in \eqref{eq:support-function-approximate-shrinker-coarse}, with this additional information on $B_{j}$ and $C$, we find
\begin{equation}\label{eq:support-function-approximate-shrinker-sharp}
\begin{split}
\bangle{F(p),\nu_{M}(F(p))} & = \bangle{p+u(p)\nu_{\Sigma}(p), A \partial_{r} + \sum_{j=1}^{n-1} B_{j} r^{-1}\partial_{\omega_{j}} + C \nu_{\Sigma}(p) }\\ 
& = A \bangle{p,\partial_{r}} + \sum_{j=1}^{n-1} B_{j} \bangle{p,r^{-1}\partial_{\omega_{j}}} + C u(p) + C\bangle{p,\nu_{\Sigma}(p)}\\
& =  A (r +O(r^{-1})) + C u(p) + 2 H_{\Sigma}(p) + O(\delta^{2-\frac 2\ell}r^{-1})\\ 
& =  C (u(p) - (r+O(r^{-1}))\partial_{r}u(p)) +  2 H_{\Sigma}(p) +  O(\delta^{2-\frac 2\ell}r^{-1})\\ 
& =  C (u(p) - r \partial_{r}u(p)) + O(r^{-1}) \partial_{r}u(p) +  2H_{\Sigma}(p)  + O(\delta^{2-\frac 2\ell}r^{-1})
\end{split}
\end{equation}

We thus have
\[
2 \phi(F(p)) + 2H_{M}(F(p)) = C( u(p) - r \partial_{r}u(p) )+ O(r^{-1}) \partial_{r}u(p) +  2H_{\Sigma}(p)  + O(\delta^{2-\frac 2\ell}r^{-1}). 
\]
Moreover, we have that (for $\ell$ sufficiently large)
\[
H_{M}(F(p)) - H_{\Sigma} =  (1+O(|\nabla u|^{2})) O(| D^{2}u| ) = O(\delta) r^{-1} + O(r^{-2}) = O(\delta) r^{-1},
\]
since $r\geq\delta^{-1}$. Thus, we find that 
\[
\partial_{r} \frac{u(p)}{r} = O(\delta) r^{-3}
\]
(assuming that $s\ll\delta$, which can be arranged since we have fixed $\delta$ independently of the value of $s$).

We now define
\[
c(\omega) : = \frac{u(\tilde R,\omega)}{\tilde R} 
\]
and observe that by interpolation of \eqref{eq:u-delta-smallness-radial-growth} with \eqref{eq:higher-der-rough-conical-u}, we have $\Vert c \Vert_{C^{2,\alpha}(\Gamma)} = O(\delta)$. Then, we set 
\[
f(r,\omega) = u(r,\omega) - c(\omega) r 
\]
We have that $f(\tilde R, \cdot) = 0$ and 
\[
\partial_{r} \frac{f(r,\omega)}{r} = O(\delta) r^{-3}
\]
Thus,
\[
r f(r,\omega) = O(\delta) (1-  r^2 \tilde R^{-2}) = O(\delta). 
\]
These two expressions imply that $\partial_{r}f = O(\delta) r^{-2}$.

Moreover, we easily see that $|D^{k}f| = O(r^{1-k})$ for $k \in \{2,\dots,\ell\}$. Interpolating this (and discarding some unnecessary decay with respect to $r$), we find that $\Vert f \Vert_{C^{2,\alpha}} = O(\delta^{1-\frac{2+\alpha}{\ell}})r^{-1}$, where the H\"older norm is taken on balls of unit size.

These estimates provide $C^{2,\alpha}_{\textrm{an},-1}$ estimates on $f$, so it remains to extend $f$ to all of $\Sigma$ while only increasing these norms by a fixed factor (we can trivially extend $c(\omega)r$). Before we do this, we must obtain improved estimates for $\partial^{2}_{r}f$. Using $C^{1}\cap C^{\ell}\subset C^{2}$ interpolation applied to the $1$-dimensional function $r\mapsto f(r,\omega)$ (for $\omega$ fixed but arbitrary), on a unit interval, we see that
\[
|\partial^{2}_{r}f(\tilde R,\omega)| \lesssim (O(\delta)\tilde R^{-2})^{1-\frac{1}{\ell-1}} \tilde R^{(1-\ell)\frac{1}{\ell-1}} = O(\delta^{1-\frac{1}{\ell-1}}) \tilde R^{\frac{2}{\ell-1} -3 }.
\]
Thus, taking $\ell$ sufficiently large, we see that 
\begin{equation}\label{eq:f-estimates-improved-second}
\tilde R |f(\tilde R,\omega)| + \tilde R^{2} |\partial_{r}f(\tilde R,\omega)| + \tilde R^{2} |\partial^{2}_{r}f(\tilde R,\omega)| = O(\delta^{\mu})
\end{equation}
for some absolute constant $\mu>0$. In particular, we emphasize that the third term in \eqref{eq:f-estimates-improved-second} is better than the $C^{2,\alpha}_{\textrm{an},-1}$ norm requires (we need this improved estimate when we extend $f$ to all of $\Sigma$). 

We now define
\[
\tilde f(r,\omega) : = 
\begin{cases}
f(r,\omega) & r\leq \tilde R\\
\partial_{r}f(\tilde R,\omega)(r-\tilde R) + \frac 1 6 \partial^{2}_{r}f(\tilde R,\omega) (\tilde R+3-r)(r-\tilde R)^{2} & r > \tilde R \end{cases}.
\]
(Recall that $f(\tilde R,\cdot) = 0$.) We then fix a cutoff function $\zeta$ with $\zeta \equiv 1$ on $(-\infty,0)$ and $\zeta\equiv 0$ on $(1,\infty)$. Then, we set $\hat f(r,\omega) = \tilde f(r,\omega)\zeta(r-\tilde R)$. Using \eqref{eq:f-estimates-improved-second}, we easily see that 
\[
\Vert \tilde f\Vert_{C^{2,\alpha}_{\textrm{an},-1}(\Sigma)} = O(\delta^{\mu}). 
\]
Thus, 
\[
\Vert (c,\tilde f) \Vert_{\cC\cS^{2,\alpha}_{-1}(\Sigma)} = O(\delta^{\mu}). 
\]
Taking $\delta$ sufficiently small depending on $\beta_{0}$, this concludes the proof. 
\end{proof}

\section{The final localized \L ojasiewicz--Simon inequality and the rough conical scale} \label{sec:final-Loj}

We now show that the error terms in the localized \L ojasiewicz--Simon inequality (Theorem \ref{theo:cutoff-loj}) are small, under the assumption that the rough conical scale is larger than the shrinker scale. 

\begin{theo}[The final localized \L ojasiewicz--Simon inequality]\label{theo:final-Loj}
Assume that $M^{n}\subset \RR^{n+1}$ has $\lambda(M)\leq \lambda_{0}$ and $\bR(M)$ sufficiently large depending on $\Sigma$. Assume that $M$ additionally satisfies the core graphical hypothesis $(*_{b,\underline r})$ and $\bR(M) \leq  \tilde \br_{\ell}(M)-1$. Then,
\[
|F(M) - F(\Sigma)| \leq C   \left( \int_{M} |\phi|^{2} \rho\, d\cH^{n} \right)^{\frac{1}{2(1-\theta/3)}}
\]
for $C=C(\Sigma,\lambda_{0},\alpha)$. Note that $\theta$ is fixed in Theorem \ref{theo:loj-entire}. 
\end{theo}
\begin{proof}
We first claim that $M$ is a roughly conical approximate shrinker up to scale $R= \Theta^{-2} \bR(M)$ in the sense of Definition \ref{defi:rough-conical-approx-shrinker}. We have already assumed that the first two conditions hold, so it remains to check that
\[
|\phi| + (1+|x|) |\nabla \phi| \leq s (1+|x|)^{-1}
\]
on $M\cap B_{\Theta R}(0)$. We will do this by modifying the proof of \cite[Corollary 1.28]{CM:uniqueness}. 

Pick $z \in M\cap B_{\Theta R}(0)$. Set $r_{z} = (1+|z|)^{-1}$, so that the Gaussian weight $\rho$ has uniformly bounded oscillation in $B_{r_{z}}(z)$. Set 
\[
\psi(z) : = \left( \int_{M\cap B_{r_{z}}({z})} |\phi_{M}|^{2} \rho \, d\cH^{n} \right)^{\frac 12} \leq e^{-\frac{\bR(M)^{2}}{8}}
\]
H\"older's inequality yields
\[
\int_{B_{r_{z}}({z})} |\phi| d\cH^{n} \lesssim r_{z}^{\frac n2} e^{\frac{|z|^{2}}{8}} \left(\int_{B_{r_{z}}(z)} |\phi|^{2} \rho \, d\cH^{n}\right)^{2} = r^{\frac n 2}_{z} e^{\frac{|z|^{2}}{8}} \psi. 
\]
Because $1+\Theta R \leq \tilde \br_{\ell}(M)$, we have that (see Definition \ref{defi:rough-conical-scale})
\[
|\nabla^{k}\phi| \leq C_{\ell} (1+|z|)^{1-k}
\]
on $M \cap B_{r_{z}}(z)$, for $k\in\{1,\dots,\ell\}$ and $z \in M\cap B_{\Theta R}(0)$. Now, by the $L^{1}\cap C^{\ell} \subset C^{0}$ interpolation inequality described in Lemma \ref{lemm:interpL1Ck}, we have that
\begin{align*}
& (1+|z|) \sup_{B_{r_{z}}(z)}|\phi| \\
& \leq C \left( r_{z}^{-1-\frac n 2} e^{\frac{|z|^{2}}{8}} \psi + \left( r_{z}^{\frac n2} e^{\frac{|z|^{2}}{8} } \psi \right)^{a_{\ell,n}} (1+|z|)^{(1-{\ell})(1-a_{\ell,n})-1} \right)\\
& \leq C \left( (1+|z|)^{\frac n 2+1} e^{\frac{|z|^{2}}{8}} \psi + (1+|z|)^{- a_{\ell,n} \frac n 2} \left( e^{\frac{|z|^{2}}{8} } \psi \right)^{a_{\ell,n}} (1+|z|)^{(1-{\ell})(1-a_{\ell,n})} \right)\\
& \leq C \left( \bR(M)^{\frac n 2 + 1} e^{-\frac{(1-{\Theta^{-2}})\bR(M)^{2}}{8}}  + e^{-\frac{a_{\ell,n}(1-{\Theta^{-2}})\bR(M)^{2}}{8}}  \right).
\end{align*}
The negative powers in the exponentials allow us to arrange that this is smaller than $s$, as long as $\bR(M)$ is sufficiently large. A similar argument can be used to bound $|\nabla \phi|$.

Thus, we see that $M$ is a roughly conical approximate shrinker up to scale $R$. Proposition \ref{prop:approx-up-to-rough} implies that the strong conical scale satisfies $\br_{\ell}(M)\geq R$. Thus, we can apply the localized \L ojasiewicz--Simon inequality from Theorem \ref{theo:cutoff-loj} to find 
\begin{align*}
|F(M) - F(\Sigma)| & \leq C \left( \left( \int_{M\cap B_{R}(0)} |\phi|^{2} \rho\, d\cH^{n} \right)^{\frac{1}{2(1-\theta)}} + R^{\frac{n-4}{2(1-\theta)}} e^{-\frac{R^{2}}{8(1-\theta)}} + e^{-\frac{R^{2}}{4\gamma}} \right)\\
& \leq C \left( \left( \int_{M} |\phi|^{2} \rho\, d\cH^{n} \right)^{\frac{1}{2(1-\theta)}} + R^{\frac{n-4}{2(1-\theta)}} e^{-\frac{R^{2}}{8(1-\theta)}} + e^{-\frac{R^{2}}{4\gamma}} \right).
\end{align*}
Note that 
\begin{align*}
R^{\frac{n-4}{2(1-\theta)}} e^{-\frac{R^{2}}{8(1-\theta)}} & = (\Theta^{-2}\bR(M))^{\frac{n-4}{2(1-\theta)}} e^{-\frac{\bR(M)^{2}}{8\Theta^{4}(1-\theta)}}\\
& = (\Theta^{-2}\bR(M))^{\frac{n-4}{2(1-\theta)}} e^{-\frac{\bR(M)^{2}}{8(1-\theta/2)}}\\
& \leq C \left( e^{-\frac{\bR(M)^{2}}{4}} \right)^{\frac{1}{2(1-\theta/3)}}\\
& = C \left(\int_{M} |\phi|^{2} \rho\, d\cH^{n}  \right)^{\frac{1}{2(1-\theta/3)}}
\end{align*}
and 
\[
e^{-\frac{R^{2}}{4\gamma}} =  \left(\int_{M} |\phi|^{2} \rho\, d\cH^{n}  \right)^{\frac{1}{\Theta^{4}\gamma}},
\]
so choosing $\gamma = 2 \Theta^{-4} (1-\theta/3) \in (1,2)$, we conclude the proof. 
\end{proof}

\section{The uniqueness of conical tangent flows: proof of Theorem \ref{thm:unique}}\label{sec:unique}

Fix $\underline r$ sufficiently large in terms of the scale of the core of the conical shrinker $\underline R(\Sigma)$, and the pseudolocality radius $\rho_{*}$ (this choice will be made explicit in Lemma \ref{lemm:rough-improves} below). 

Now, fix $\eps=\eps(\Sigma,\underline r)>0$\index{$\eps$} will be chosen sufficiently small below. Suppose that $\{M_{\tau}\}_{\tau\in[-1,\infty)}$\index{$M_{\tau}$} is a rescaled mean curvature flow (Brakke flow) on $[-1,\infty)\times \RR^{n+1}$ so that there is 
\[
u : (\Sigma\cap B_{\eps^{-1}}(0)) \times [-1,\eps^{-2}) \to \RR
\]
with 
\begin{enumerate}
\item $\Graph u(\cdot, \tau) \subset M_{\tau}$
\item $M_{\tau}\cap B_{\eps^{-1}-1} \subset \Graph u(\cdot,\tau)$. 
\item $\Vert u \Vert_{C^{\ell+1}(\Sigma\cap B_{\eps^{-1}}(0))} \leq \eps$, and 
\item $F(M) - F(\Sigma) \leq \eps$.
\end{enumerate}
Here, $\ell \in \NN$ controls the number of derivatives in the definition of $\tilde \br_{\ell}$. It has been fixed in Proposition \ref{prop:approx-up-to-rough}. We additionally fix $\lambda_{0}$ so that $\lambda(M_{0}) \leq \lambda_{0}$ (which implies that $\lambda(M_{\tau}) \leq \lambda_{0}$). Finally, we assume that there is a sequence if times $s_{k}\to \infty$ so that $M_{s_{k}}$ converges smoothly on compact subsets of $\RR^{n+1}$ to $\Sigma$ (with multiplicity one).

Recall that the core graphical hypothesis $(*_{b,\underline r})$ has been defined in Definition \ref{def:core-graph-hypoth}. Define the \emph{graphical time}\index{graphical time} $\bar \tau$\index{$\bar \tau$} by 
\[
\bar \tau : = \sup\{ \hat \tau \in [-1,\infty) : M_{\tau} \textrm{ satisfies } (*_{b,\underline r}) \textrm{ for all } \tau \in [-1,\hat \tau]\}. 
\]
Our first goal is to show that $\bar \tau = \infty$. Note that by taking $\eps$ sufficiently small (depending on $b,\underline r,\Sigma$, we we can assume that $\bar \tau$ is arbitrarily large.

\begin{lemm}[The rough conical scale improves rapidly] \label{lemm:rough-improves} There is $\underline r_{0}(\Sigma,\underline R(\Sigma),\rho_{*})$ sufficiently large so that taking $\underline r\geq{\underline{ r}_{0}}$, $\eps_{0} = \eps_{0}(\Sigma,\underline r)$ sufficiently small, and fixing $C_{\ell}=C_{\ell}(\Sigma,\underline r)$ sufficiently large in the definition of the rough shrinker scale,
we have that $\tilde \br_{\ell}(M_{\tau})\geq \frac 12 e^{\frac\tau 2}\underline r$ for all $\tau \in [0,\bar \tau)$. 

Moreover, we can find $u : \Sigma \cap B_{4\underline r}(0) \times [0,\overline \tau) \to \RR$ with $u(\cdot,\tau)$ uniformly bounded in $C^{\ell+2}$ and with
\[
\Graph u (\cdot,\tau) \subset M_{\tau}\qquad \textnormal{and} \qquad M_{\tau} \cap B_{4\underline r-1}\subset \Graph u (\cdot,\tau) . 
\]
\end{lemm}
\begin{proof}
Consider $\tau_{0} \in [0,\bar \tau)$. Note that $t\mapsto \sqrt{-t} M_{(\tau_{0}-\log(-t))} : = \hat M_{t}^{(\tau_{0})}$ is a mean curvature flow for $t \in [-e^{\tau_{0}},0)$ and $\hat M_{-1}^{(\tau_{0})} = M_{\tau_{0}}$. Take $b$ sufficiently small in the core graphical hypothesis. Then, by definition of the {pseudolocality scale $\rho_*$, the scale  $\underline R(\Sigma)$ of the core of the shrinker and the core graphical scale $\underline r$}, we can ensure that for 
\[
x \in \hat M_{-1}^{(\tau_{0})}\cap (\overline{ B_{\underline r-2\rho_{*}}(0)}\setminus B_{\underline R(\Sigma)}(0)),
\]
there is some plane $\Pi_{x}$ through $x$ so that $\hat M_{-1}^{(\tau_{0})}\cap B_{\rho_{*}}(x)$ is a Lipschitz graph over $\Pi_{x}$ with Lipschitz constant at most $\gamma_{*}$. Thus, by pseudolocality (Proposition \ref{prop:pseudolocality-result}), $\hat M_{t}^{(\tau_{0})}\cap B_{\delta_{*}}(x)$ is non-empty, and a $\delta_{*}$-Lipchitz graph over $\Pi_{x}$ for all $t \in [-1,0)$. 

We can patch these graphs together to write find a family of domains ${\hat{\Omega}_t^{(\tau_0)}}\subset \Sigma$ with 
\[
(\Sigma \cap (\overline{ B_{\underline r-3\rho_{*}}(0)}\setminus B_{\underline R(\Sigma)+1}(0)))\subset \hat \Omega_{t}^{(\tau_{0})}
\]
and a function $\hat v_{t}^{(\tau_{0})} : \Omega_{t}\to\RR$ so that $\Graph \hat v_{t}^{(\tau_{0})}  \subset \hat M_{t}^{\tau_{0}}$. Using a shrinking sphere as a barrier, we can see that for $t \in [-1,0)$, this graph describes all of $\hat M_{-1}^{(\tau_{0})}\cap (\overline{ B_{\underline r-4\rho_{*}}(0)}\setminus B_{\underline R(\Sigma)+2}(0))$. The shrinking sphere of radius $\underline r-4\rho_{*}$ at $t=-1$ still contains $B_{\underline r-3\rho_{*}}(0)$ as long as we choose $\underline r$ sufficiently large so that
\[
(\underline r - 4\rho_{*})^{2} \leq (\underline r - 3\rho_{*})^{2} - 1\qquad \Leftrightarrow \qquad   \frac{1}{2\rho_{*}}(5 \rho_{*}^{2} + 1) \leq \underline r . 
\]
Now, for $\omega \in (0,1)$, by applying interior estimates \cite{EckerHuisken} (cf.\ \cite[Corollary 8.4]{BegleyMoore}) for graphical mean curvature flow, we find that
\[
|\nabla^{(k)}A_{\hat M^{(\tau_{0})}_{t}}|(x) \leq C=C(\Sigma,\lambda_{0},\omega)
\]
for $x \in \hat M_{-t}^{(\tau_{0})}\cap (\overline{ B_{\underline r - 5\rho_{*}}(0)}\setminus B_{\underline R(\Sigma)+3}(0))$, $t \in [-1+\omega,0)$, and $k \in \{0,\dots,\ell\}$. 

By the definition of the core graphical scale, $\tilde\br_{\ell}(M_{\tau})\geq \underline r$, so the desired curvature estimates hold on $M_{\tau}\cap B_{\underline r}(0)$. Moreover, by taking the parameter $\eps$ sufficiently small, we can ensure that the desired estimates hold for $\tau \in [-1,1]$. On the other hand, for $\tau \in [1,\bar \tau)$ and 
\[
x \in M_{\tau} \cap \left(\overline{B_{e^{\frac \tau 2}(\underline r-5\rho_{*})}(0)} \setminus B_{\underline r}(0) \right), 
\]
We choose 
\[
\tau_{0} = \tau + 2 \log (|x|^{-1} (\underline r - 5\rho_{*})) \in [0,\tau)
\]
Then,
\[
t = -e^{\tau_{0}-\tau} = - |x|^{-2} (\underline r - 5\rho_{*})^{2} \in [-1+\omega,0),
\]
for $\omega = \omega(n,\underline r) \in (0,1)$ fixed by
\[
\omega : = 1-(1-5\underline r^{-1}\rho_{*})^{2} \in (0,1).
\]
Now, we find that the point $x$ is rescaled to
\[
\hat x : = \sqrt{-t}x \in \hat M_{-t}^{(\tau_{0})}\cap \partial  B_{\underline r - 5\rho_{*}}(0),
\]
so the curvature estimates established above yield
\[
|\hat x|^{-1-k} |\hat x|^{1+k}|\nabla^{(k)}A_{\hat M^{(\tau_{0})}_{t}}|(\hat x) = |\nabla^{(k)}A_{\hat M^{(\tau_{0})}_{t}}|(\hat x) \leq C=C(\Sigma,\lambda_{0},\underline r)
\]
Unwinding this, we find
\[
|x|^{1+k}|\nabla^{(k)}A_{M_{\tau}}|( x)  \leq C(\underline r - 5 \rho_{*})^{1+k},
\]
for $k\in \{0,\dots,\ell\}$. Thus, by choosing $C_{\ell} = C_{\ell}(\Sigma,\lambda_{0},\underline r)$ sufficiently large, we find that $\tilde\br_{\ell}(M_{\tau}) \geq e^{\frac\tau 2}(\underline r - 5\rho_{*})$, as claimed. As such, the asserted curvature estimates follow by requiring that $\underline r \geq 10 \rho_{*}$.

The above proof also shows that there is a function $u : (\Sigma \cap B_{4\underline r}) \times [0,\overline\tau) \to \RR$ with 
\[
\Graph u (\cdot,\tau)  \subset M_{\tau} \qquad \text{and} \qquad M_{ \tau}\cap B_{4\underline r-1}\subset \Graph u (\cdot,\tau) ,
\]
 and so that $u(\cdot,\tau)$ uniformly bounded in $C^{\ell+2}$. Note that this $u$ agrees with the function in the definition of the core graphical hypothesis, on their common domain of definition. 
\end{proof}

First, suppose that $\tau$ is such that $\bR(M_{\tau})\leq \tilde \br_{\ell}(M_{\tau}) - 1$. By Theorem \ref{theo:final-Loj}, we have that for $\theta' = \theta/3$\index{$\theta'$},
\[
F(M_{\tau}) - F(\Sigma) \leq C   \left( \int_{M_{\tau}} |\phi|^{2} \rho\, d\cH^{n} \right)^{\frac{1}{2(1-\theta')}},
\]
so
\begin{align*}
-\frac{d}{d\tau} (F(M_{\tau}) - F(\Sigma))^{\theta'} & = \theta'  (F(M_{\tau}) - F(\Sigma))^{\theta'-1} \int_{M_{\tau}} |\phi|^{2} \rho\, d\cH^{n} \\
& \geq C  \left( \int_{M_{\tau}} |\phi|^{2} \rho\, d\cH^{n} \right)^{\frac 12}.
\end{align*}
On the other hand, suppose that $\tau$ is such that $\bR(M_{\tau}) > \tilde \br_{\ell}(M_{\tau}) - 1 \geq \frac 12 e^{\frac \tau 2}\underline r - 1$ (by Lemma \ref{lemm:rough-improves}). The following coarse estimate will suffice in this case:
\begin{equation}\label{eq:huge-shrinker-scale-trivial-est}
\left( \int_{M_{\tau}} |\phi|^{2} \rho\, d\cH^{n} \right)^{\frac 12} = e^{-\frac{\bR(M_\tau)^{2}}{8}} \leq C e^{-\tau}.
\end{equation}

Thus, we can conclude that for all $\tau \in [0,\overline \tau)$
\begin{align*}
C  \left( \int_{M_{\tau}} |\phi|^{2} \rho\, d\cH^{n} \right)^{\frac 12} \leq -\frac{d}{d\tau} (F(M_{\tau}) - F(\Sigma))^{\theta'} + e^{-\tau}. 
\end{align*}
Integrating this, we find that for $\tau_{0} \in [0,\overline \tau)$,
\begin{equation}\label{eq:conseq-loj}
\int_{\tau_{0}}^{\overline \tau} \left( \int_{M_{\tau}} |\phi|^{2} \rho\, d\cH^{n} \right)^{\frac 12} d\tau \lesssim (F(M_{\tau_{0}}) - F(\Sigma))^{\theta'} + e^{-\tau_{0}} \lesssim \eps^{\theta'} + e^{-\tau_{0}}
\end{equation}

For the function $u: (\Sigma\cap B_{2\underline r})\times [0,\overline \tau) \to \RR$ described in Lemma \ref{lemm:rough-improves},  we have that
\[
 \int_{M_{\tau}} |\phi|^{2} \rho\, d\cH^{n}  \geq C \left\Vert \frac{\partial u}{\partial \tau} \right \Vert_{L^{2}(\Sigma\cap B_{4\underline r}(0))}^{2},
\]
so we see that 
\[
\sup_{\tau \in [\tau_{0},\overline \tau)} \left\Vert u(\cdot,\tau) - u(\cdot,\tau_{0}) \right \Vert_{L^{2}(\Sigma\cap B_{4\underline r}(0))} \lesssim \eps^{\theta'} + e^{-\tau_{0}}.  
\]
Because $u(\cdot,\tau)$ is uniformly bounded in $C^{\ell+2}$ by Lemma \ref{lemm:rough-improves}, by taking $\eps$ sufficiently small and $\tau_{0} = \frac 12 \eps^{-2}$, we have that 
\[
\Vert u (\cdot,\tau_{0})\Vert_{C^{\ell+1}(\Sigma\cap B_{2\underline r}(0))} \leq \frac b 4
\] 
and 
\[
\sup_{\tau \in [\tau_{0},\overline \tau)} \left\Vert u(\cdot,\tau) - u(\cdot,\tau_{0}) \right \Vert_{C^{\ell+1}(\Sigma\cap B_{2\underline r}(0))} \leq \frac b 4.
\]
Thus, we see that $\Vert u(\cdot,\tau)\Vert_{C^{\ell+1}(\Sigma\cap B_{2\underline r}(0))} \leq \frac b 2$ for $\tau \in [0,\overline \tau)$. This (combined with pseudolocality and interior estimates) implies that we can extend the graphical hypothesis slightly beyond $\overline \tau$, a contradiction.

Thus, $\overline \tau = \infty$. Now, returning to \eqref{eq:conseq-loj}, we have that (recall that $s_{k}\to\infty$ are so that $M_{s_{k}}\to\Sigma$)
\[
\sup_{\tau \in [s_{k},\infty)} \left\Vert u(\cdot,\tau) - u(\cdot,s_{k}) \right \Vert_{L^{2}(\Sigma\cap B_{4\underline r}(0))} \lesssim  \int_{s_{k}}^{\infty}  \left\Vert \frac{\partial u}{\partial \tau} \right \Vert_{L^{2}(\Sigma\cap B_{4\underline r}(0))}  d\tau \lesssim (F(M_{s_{k}})-F(\Sigma))^{\theta'} + e^{-s_{k}} 
\]
Since $u(\cdot,s_{k})\to 0$ in $L^{2}(\Sigma\cap B_{4\underline r}(0))$, we thus see that $u(\cdot,\tau) \to 0$ in $L^{2}(\Sigma \cap B_{4\underline r}(0))$ as $\tau \to\infty$, and thus in $C^{\ell+1}(\Sigma \cap B_{2\underline r}(0))$. 

From this, it is clear that $M_{\tau}$ converges on compact sets to $\Sigma$ as $\tau\to\infty$. This completes the proof of Theorem \ref{thm:unique}.

\subsection{Rate of convergence} \label{subsec:rate-convergence}
Here, we observe that similar arguments can yield a rate of convergence of $M_{\tau}$ towards $\Sigma$. Arguing as above, we have that 
\begin{align*}
\frac{d}{d\tau} (F(M_{\tau}) - F(\Sigma)) & = -\int_{M_{\tau}} |\phi|^{2} \rho \, d\cH^{n}\\
& \leq - C(F(M_{\tau}) - F(\Sigma))^{2(1-\theta')} + Ce^{-2\tau}
\end{align*}
for all $\tau \in [0,\infty)$.
We claim that 
\[
F(M_{\tau}) - F(\Sigma) \leq D (1+\tau)^{-\frac{1}{1-2\theta'}}.
\]
for $D$ sufficiently large in terms of $M_{0}$ and $\Sigma$. Indeed, letting $\tilde\tau$ denote the first time this fails, since $e^{-2x} \lesssim (1+x)^{-\alpha}$ for all $x>0$, we have that
\[
(F(M_{\tilde\tau}) - F(\Sigma))^{2(1-\theta')} = D^{2(1-\theta')} (1+\tilde \tau)^{-\frac{2(1-\theta')}{1-2\theta'}} \geq c D^{2(1-\theta')} e^{-2\tilde \tau}
\]
Thus, as long as $D$ sufficiently large, we find that 
\begin{align*}
- \frac{ D}{1-2\theta'} (1+\tilde \tau)^{-\frac{2(1-\theta')}{1-2\theta'}} & \leq \frac{d}{d\tau} (F(M_{\tau}) - F(\Sigma))\Big|_{\tau = \tilde \tau} \\
& \leq - C(F(M_{\tilde \tau}) - F(\Sigma))^{2(1-\theta')} \\
& \leq - C  D^{2(1-\theta')}(1+\tilde \tau)^{-\frac{2(1-\theta')}{1-2\theta'}}
\end{align*}
Taking $D$ larger if necessary, this yields a contradiction. Thus, we have that for any $R$ fixed,
\begin{align*}
\Vert u(\cdot,\tau) \Vert_{L^{2}(\Sigma\cap B_{2R})} & \lesssim \int_{\tau}^{\infty} \left( \int_{M_{\tau}} |\phi|^{2} \rho\, d\cH^{n} \right)^{\frac 12} d\tau \\
& \lesssim (F(M_{\tau}) - F(\Sigma))^{\theta'} + e^{-\tau}\\
& \lesssim  (1+\tau)^{-\frac{\theta'}{1-2\theta'}}
\end{align*}
Interpolating yields 
\[
\Vert u(\cdot, \tau) \Vert_{C^{k}(\Sigma\cap B_{R})} \lesssim (1+\tau)^{-\frac{\theta'}{1-2\theta'}+\eta}
\]
for any $k$, $R$, and $\eta>0$, as $\tau\to\infty$. 

{\subsection{Proof of Corollary \ref{coro:conical-struct}} \label{subsec: cor-proof}
 Note that the proof of Theorem \ref{thm:unique} implies that there exists $\eps>0$ such that the surfaces $M_t \cap B_{\eps}(0)$ for $t \in (-\eps^2, 0)$ are smooth graphs over $\sqrt{t}\cdot \Sigma$. Even more, one also sees that $(M_0 \cap B_{\eps}(0))\setminus \{0\}$ is a smooth normal graph over the asymptotic cone $C$ of $\Sigma$ with curvature bounded by $c/r$ (plus all corresponding higher order derivative estimates). Note that the tangent flow $\cM_\Sigma$ has as the time zero slice the cone $C$. Thus by taking rescaling limits of the flow, including time zero, we see that the uniqueness of the tangent flow implies that rescalings of $M_0$ converge smoothly on compact subsets of $\RR^n \setminus \{0\}$ to $C$.}
 
\appendix

\section{Standard definitions}\label{app:defi}
We recall the following definitions and conventions:

\begin{defi}\label{defi:gaussian-area}
For $M^{n}\subset \RR^{n+1}$ with polynomial area growth, the \emph{Gaussian area} of $M$ is \index{$F(\cdot)$}
\[
F(M) = \int_{M} \rho \, d\cH^{n}
\]
where $\rho = (4\pi)^{-\frac n 2} e^{-|x|^{2}/4}$\index{$\rho$}. {Recall, that the entropy $\lambda(M)$ \index{$\lambda(M)$} is defined as the supremum of the Gaussian area over all centers and scales, see \cite{CM:generic}. }

\end{defi}

\begin{defi}
A hypersurface $\Sigma^{n}\subset \RR^{n+1}$ is a \emph{self-shrinker} if $\sqrt{-t}\cdot\Sigma$ is a solution to mean curvature flow for $t \in (-\infty,0)$. This is equivalent to 
\begin{equation}\label{eq:self-shrinker}
H_{\Sigma} = \frac 12 \bangle{x,\nu_{\Sigma}}
\end{equation}
\end{defi}

\begin{defi}\label{defi:phi}
For a general hypersurface $M^{n}\subset \RR^{n+1}$, we define the function
\[
\phi = \phi_{M} := \frac 12 \bangle{x,\nu_{M}} - H_{M}. 
\]
\index{$\phi$}
Note that $\Sigma$ is a self-shrinker if and only if $\phi_{\Sigma} \equiv 0$. 
\end{defi}

\begin{defi}\label{defi:as-con}
A smooth self-shrinker $\Sigma^{n}\subset \RR^{n+1}$ is (smoothly) \emph{asymptotically conical} if
\[
\lim_{t\nearrow 0}\sqrt{-t}\cdot\Sigma = \cC
\]
in $C^{\infty}_{\textrm{loc}}(\RR^{n+1}\setminus]\{0\})$ with multiplicity one, where $\cC$ is a cone over a smooth closed hypersurface $\Gamma^{n-1}\subset \SS^{n}\subset \RR^{n+1}$. 
\end{defi}

\begin{defi}\label{defi:L-operators}
We define the following operators along $\Sigma$:
\begin{align*}
\cL_{\gamma} u & : = \Delta_{\Sigma} u - \frac 12 \vec{x}\cdot \nabla_{\Sigma} u + \gamma u \\
L u & := \cL_{\frac 12}u + |A_{\Sigma}|^{2}u = \Delta_{\Sigma} u - \frac 12 (\vec{x}\cdot \nabla_{\Sigma} u - u) + |A_{\Sigma}|^{2} u .
\end{align*}
\index{$\cL_{\gamma}$}\index{$L$}\index{$\cL_{0}$}\index{$\cL_{\frac12}$}Note that $L$ is the full second variation of Gaussian area along $\Sigma$. Moreover, $\cL_{0}$ and $\cL_{\frac 12}$ will be particularly relevant in the sequel. 
\end{defi}

\section{Interpolation inequalities} \label{app:interpolate}
We recall the following standard interpolation inequalities in multiplicative form. 
\begin{lemm}\label{lemm:interpolation}
Suppose that $u \in C^{k}(B_{2})$, then for $j< k$,
\[
\Vert D^{j} u \Vert_{C^{0}(B_{1})} \leq C \Vert u\Vert_{C^{0}(B_{2})}^{1-\frac{j}{k}} \Vert D^{k} u \Vert_{C^{0}(B_{2})}^{\frac j k}
\]
for $C=C(n,k)$. Similarly, if $u \in C^{k,\alpha}(B_{2})$, then for $j+\beta < k+\alpha$,
\[
[ D^{j} u ]_{\beta;B_{1}} \leq C \Vert u\Vert_{C^{0}(B_{2})}^{1-\frac{j+\beta}{k+\alpha}} [ D^{k} u ]_{\alpha ; B_{2}}^{\frac{j+\beta}{k+\alpha}}
\]
for $C=C(n,k,\alpha,\beta)$.
\end{lemm}

These follow in a similar manner to the linear inequalities given in \cite[Lemma 6.32]{giltru}, except in the proof one should optimize with respect to the parameter $\mu$ rather than just choosing $\mu$ sufficiently small. Alternatively, see \cite[Lemma A.2]{Hormander:geodesy}.

We will also need the following interpolation inequality. 
\begin{lemm}[{cf.\ \cite[Lemma B.1]{CM:uniqueness}}]\label{lemm:interpL1Ck}
If $u$ is a $C^{k}$ function on $B_{2r}\subset \RR^{n}$ then
\begin{align*}
\Vert u \Vert_{L^{\infty}(B_{r})} & \leq C\left( r^{-n} \Vert u\Vert_{L^{1}(B_{2r})} + \Vert u\Vert_{L^{1}(B_{2r})}^{a_{k,n}} \Vert \nabla^{k}u\Vert_{L^{\infty}(B_{2r})}^{1-a_{k,n}}\right)\\
r \Vert \nabla  u \Vert_{L^{\infty}(B_{r})} & \leq C\left( r^{-n} \Vert u\Vert_{L^{1}(B_{2r})} + \Vert u\Vert_{L^{1}(B_{2r})}^{b_{k,n}} \Vert \nabla^{k}u\Vert_{L^{\infty}(B_{2r})}^{1-b_{k,n}}\right)
\end{align*}
for $C=C(k,n)$, $a_{k,n}=\frac{k}{k+n}$, and $b_{k,n} = \frac{k-1}{k+n}$.
\end{lemm}

\section{Geometry of normal graphs}\label{app:NG}
We consider hypersurfaces $M,N$ in $\RR^{n+1}$ such that $N$ can be locally written as a normal graph over $M$ with height function $v$, where we assume that the $C^1$-norm of $u$ is sufficiently small (depending on the geometry of $M$). Let $p \in M$ and choose a local parametrisation $F$, parametrising an open neighbourhood $U$ of $p$ in $M$ such that $F(0)=p$. We can assume that $ g_{ij}=\bangle{\partial_i F, \partial_j F}$ satisfies
$$ g_{ij}\big|_{x=0} = \delta_{ij}  \text{ and } \partial_kg_{ij}|_{x=0} = 0\, .$$
For simplicity we can furthermore assume that the second fundamental form $(h_{ij})$ is diagonalised at $p$ with eigenvalues $\lambda_1,\ldots, \lambda_n$. A direct calculation, see \cite[(2.27)]{Wang:uniqueness-conical}, yields that the normal vector $\nu_N(q)$, where $q = p + u(p) \nu_M(q)$, is co-linear to the vector
$$ N = - \sum_{i=1}^n \frac{\partial_i u}{1 - \lambda_i u} \partial_iF\bigg|_{x=0} + \nu_M(p)\ . $$
Denoting the shape operator by $S = (h^i_{\ j})$ we see that thus in coordinate free notation
\begin{equation}\label{eq:ng.1}
\nu_N(q) = v^{-1} \left( - (\text{Id} - uS)^{-1}\nabla^M u + \nu_M\right) (p)\, ,
\end{equation}
where $v:=(1 + |(\text{Id} - uS)^{-1}(\nabla^M u)|^2)^{\frac{1}{2}}$.
This implies
\begin{equation}\label{eq:ng.2}
\langle q, \nu_N(q) \rangle = v^{-1} \left( u + \langle p, \nu_M(p) \rangle - \big\langle p, (\text{Id} - uS)^{-1}\nabla^M u \big\rangle \right)\, .
\end{equation}
For the induced metric $\tilde g$ one obtains in the above coordinates at $p$, again see \cite[(2.32)]{Wang:uniqueness-conical},
$$ \tilde g_{ij} = (1-\lambda_iu)(1-\lambda_ju) \delta_{ij} + \partial_i u \partial_j u$$
which implies
\begin{equation}\label{eq:ng.3}
\tilde g^{ij} = \frac{\delta^{ij}}{(1-\lambda_i u)(1-\lambda_j u)} -  v^{-2}\frac{\partial_i u}{(1-\lambda_i u)^2}\frac{\partial_j u}{(1-\lambda_j u)^2}\, .
\end{equation}
Furthermore, from \cite[(2.30)]{Wang:uniqueness-conical} we have
\begin{equation}\label{eq:ng.4}
\begin{split}
\tilde h_{ij} = \langle \partial^2_{ij} \tilde F,  \nu_N \rangle &= v^{-1} \Big( \frac{\lambda_i}{1-\lambda_i u} \partial_i u \partial_ju + \frac{\lambda_j}{1-\lambda_j u} \partial_i u \partial_ju\\
&\qquad\quad\ + \sum_k \frac{u}{1-\lambda_k u}\, \partial_k u \,\partial_i h_{jk} + h_{ij} - \lambda_i\lambda_j u\, \delta_{ij} + \partial^2_{ij} u\Big)\, ,
\end{split}
\end{equation}
which yields a closed expression for the mean curvature $\tilde H$ of $N$, since $\tilde H(p) = \tilde g^{ij}(p) \tilde h_{ij}(p)$.

\section{Variations of Gaussian area and the Euler--Lagrange equation}\label{app:EL}
Suppose that $M^{n}\subset \RR^{n+1}$ is a normal graph of $v : \Sigma\to\RR$ for a fixed shrinker $\Sigma^{n}$. Recall that the Gaussian area is defined as
\[
F(M) : = \int_{M}\rho\, d\cH^{n}. 
\]
For $\tilde v$ a variation of $v$ (i.e., a variation in the normal direction to $\Sigma$), the first variation of $F$ in the direction of $\tilde v$ satisfies (see \cite{Schulze:compact}) 
\begin{align*}
\delta_{\tilde v}F(M) & = - \int_{\Sigma} \Pi_{T^\perp\Sigma}\left(\vec{H}_{M} + \frac{x^{\perp}}{2} \right)\Big|_{x=y+v(y)\nu_{\Sigma}(y)}\tilde v(y) J(y,v,\nabla_{\Sigma}v)\rho(y+\nu_{\Sigma})\rho(y)^{-1} \rho(y) d\cH^{n},
\end{align*}
where $\Pi_{T^\perp\Sigma}$\index{$\Pi_{T^\perp\Sigma}$} is the projection on to the normal bundle to $\Sigma$ and \index{$J$}
\[
J(y,v,\nabla_{\Sigma}v) = \Jac(D \exp_{y}(v(y)\nu_{\Sigma}(y)))
\]
 is the area element. 

Hence, the Euler--Lagrange operator $\cM$\index{$\cM$} (with respect to the weighted space $L^{2}_{W}$) satisfies 
\[
\cM(v) =  \Pi_{T^\perp\Sigma}\left(\vec{H}_{M} + \frac{x^{\perp}}{2} \right)\Big|_{x=y+v(y)\nu_{\Sigma}(y)}  J(y,v,\nabla_{\Sigma}v)\rho(y+v(y)\nu_{\Sigma})\rho(y)^{-1} 
 \]
It is well known that the linearization of $\cM$ at $v=0$ is the $L$ operator (cf.\ \cite[Lemma 4.3]{CM:uniqueness}). 

\section{Area growth bounds from Gaussian area estimates} \label{app:area-bds}
The following is a well known fact:
\begin{lemm}\label{lemm:poly-area-growth}
For $M^{n}\subset \RR^{n+1}$ with $\lambda(M) \leq \lambda_{0}$, there is $C=C(\lambda_{0},n)$ so that 
\[
\cH^{n}(M\cap B_{R}(x)) \leq C R^{n}
\]
for all $R>0$ and $x \in \RR^{n+1}$. 
\end{lemm}

\printindex

\bibliographystyle{amsalpha}
\bibliography{bib}

\end{document}